\documentclass[11pt]{article}
\usepackage[utf8]{inputenc}
\usepackage{geometry}
\usepackage{amssymb}
\usepackage{amsthm}
\usepackage{amsmath}
\usepackage{enumerate}
\usepackage{booktabs}
\usepackage{graphicx}
\usepackage{multirow}
\usepackage{color}
\usepackage{soul}
\usepackage{authblk}

\usepackage{hyperref}
\usepackage[capitalise]{cleveref}

\usepackage{amsfonts}
\usepackage{graphicx}
\usepackage{epstopdf}
\usepackage{booktabs}
\usepackage{multirow}
\usepackage{bm}

\newcommand{\norm}[1]{\left\lVert#1\right\rVert}
\newcommand{\normt}[1]{\left\lVert#1\right\rVert_2}
\newcommand{\lip}[1]{\left[#1\right]}
\newcommand{\brac}[1]{\left(#1\right)}
\DeclareMathOperator{\R}{\mathbb{R}}
\newcommand{\C}{\mathbb{C}}
\newcommand{\N}{\mathbb{N}}
\newcommand{\D}{\mathbf{D}}
\newcommand{\K}{\mathcal{K}}
\newcommand{\M}{\mathcal{M}}
\newtheorem{theorem}{Theorem}
\newtheorem{corollary}{Corollary}
\newtheorem{lemma}{Lemma}
\newtheorem{definition}{Definition}

\newtheorem{example}{Example}
\newtheorem{assumption}{Assumption}

\title{Error analysis based on inverse modified differential equations for discovery of dynamics using linear multistep methods and deep learning}

\author[1]{Aiqing Zhu}
\author[1]{Sidi Wu}
\author[1*]{Yifa Tang}
\affil[1]{LSEC, ICMSEC, Academy of Mathematics and Systems Science, Chinese Academy of Sciences, Beijing 100190, China}
\affil[*]{Email address: tyf@lsec.cc.ac.cn}
\date{}

\begin{document}

\maketitle

\begin{abstract}
Along with the practical success of the discovery of dynamics using deep learning, the theoretical analysis of this approach has attracted increasing attention. Prior works have established the grid error estimation with auxiliary conditions for the discovery of dynamics using linear multistep methods and deep learning. And we extend the existing error analysis in this work. We first introduce the concept of inverse modified differential equations (IMDE) for linear multistep methods and show that the learned model returns a close approximation of the IMDE. Based on the IMDE, we prove that the error between the discovered system and the target system is bounded by the sum of the LMM discretization error and the learning loss. Furthermore, the learning loss is quantified by combining the approximation and generalization theories of neural networks, and thereby we obtain the priori error estimates for the discovery of dynamics using linear multistep methods. Several numerical experiments are performed to verify the theoretical analysis.
\end{abstract}

\begin{keywords}
Learning dynamics, data-driven discovery, linear multistep methods, deep learning, backward error analysis
\end{keywords}


\section{Introduction}
The discovery of hidden dynamical systems from observed data is a significant and challenging task across diverse scientific fields \cite{brunton2019data,brunton2016discovering,schmidt2009distilling}, and also plays an important role in various applications of machine learning, such as robotic manipulation \cite{hersch2008dynamical} and autonomous driving \cite{levinson2011towards}. Besides some popular approaches such as symbolic regression \cite{bongard2007automated,schmidt2009distilling}, sparse regression \cite{brunton2016discovering, rudy2017data}, Gaussian process \cite{raissi2017machine,kocijan2005dynamic} and Koopman theory \cite{brunton2017chaos}, numerous approaches leverage deep learning to discover mathematical models from data. We refer the reader to \cite{du2022discovery} for a comprehensive overview of techniques for the discovery of dynamics. There were multiple pioneering efforts using neural networks to model continuous time-series data in the 1990s \cite{anderson1996comparison,gonzalez1998identification,rico1994continuous,rico1993continuous}. In recent years, a growing number of models combining numerical integrators and neural networks have been developed and applied to learn hidden dynamics expressed by differential equations \cite{chen2018neural,hu2022revealing, kolter2019learning,qin2019data,raissi2018multistep}. In addition, such learning models have been further extended to incorporate the physical inductive bias of the underlying problems \cite{bertalan2019learning,botev2021priors, greydanus2019hamiltonian, huh2020time, lu2022modlanets, yu2021onsagernet,zhang2022gfinns,zhong2021extending}.

With satisfactory progress in the discovery using deep learning, our theoretical understanding of it is also being developed \cite{du2022discovery,keller2021discovery,zhu2022on}. In this work, we focus on the discovery of dynamics using linear multistep methods (LMMs) and deep learning, which we refer to as LMNets \cite{du2022discovery, keller2021discovery,raissi2018multistep, xie2019non}. This technique is essentially an inverse process of identifying the unknown governing function $f$ of a dynamical system using the provided information of the flow map at given phase points. Our aim is to extend the prior existing error estimations for such a learning model, as presented in \cite{du2022discovery, keller2021discovery}, from a new perspective.

To begin with, we briefly review the data-driven discovery models LMNets following \cite{du2022discovery,keller2021discovery,raissi2018multistep}. Suppose the vector-valued functions $y(t): \R\rightarrow \R^D$ and $f(y): \R^D\rightarrow \R^D$ obey the differential equation
\begin{equation}\label{eq: ode0}
\frac{d}{dt}y(t) = f(y(t))
\end{equation}
but are both unknown. The objective is to find a closed-form expression for the unknown target governing function $f$ with given $y_n = y(nh)$ for $n=0,\cdots, N$. By applying an LMM, we can build the discrete relation between the states $y_n$ and the approximations $f_n \approx f(y_n)$ \cite{du2022discovery,keller2021discovery,raissi2018multistep}, i.e., 
\begin{equation}\label{eq: relation of lmnets}
\sum_{m=0}^M \alpha_my_{n+m} - h\sum_{m=0}^M\beta_mf_{n+m}=0, \quad n = 0, \cdots, N-M.
\end{equation}
To obtain closed-form expressions, we employ a neural network, denoted by $f_{\theta}$, to approximate $f$. The parameter of $f_{\theta}$ can be obtained by optimizing the training loss obtained by replacing each occurrence of $f_{n+m}$ in (\ref{eq: relation of lmnets}) with $f_{\theta}(y_{n+m})$,
\begin{equation}\label{eq: loss of lmnets}
\mathcal{L}_{h,N} (f_{\theta})= \frac{1}{N-M+1}\sum_{n=0}^{N-M}\norm{\sum_{m=0}^M h^{-1}\alpha_my_{n+m} - \sum_{m=0}^M\beta_mf_{\theta}(y_{n+m})}_2^2.
\end{equation}
After optimization, the unknown $f$ is recovered with an explicit expression, and we can then predict the dynamics on nearby data by solving (\ref{eq: ode0}) by means of the generalization of deep learning.

Although the $f_n$ might not be all involved in (\ref{eq: relation of lmnets}), we typically have more unknowns than we have equations. Thus the linear system (\ref{eq: relation of lmnets}) has infinitely many solutions, and $\mathcal{L}_h$ admits infinitely many global minimizers that lead to near-zero loss but can be radically different from the target governing function. In \cite{keller2021discovery}, the auxiliary conditions are proposed to ensure uniqueness, and a framework based on refined notions is established for rigorous convergence and stability analysis. Until recently, the best error estimation was due to Du et al. \cite{du2022discovery}. By introducing an augmented loss function based on the auxiliary conditions, they prove that the gird error of the global minimizer is bounded by the sum of the discretization error and the approximation error, yielding that LMNets are able to have high error orders.

However, prior existing analyses \cite{du2022discovery,keller2021discovery} have primarily focused on training networks with auxiliary conditions, and have only quantified the error evaluated at the given sample locations. In practice, the use of deep learning approaches introduces further implicit regularization. Du et al. \cite{du2022discovery} numerically find that LMNets without the auxiliary conditions still perform well as long as the time step size is sufficiently small. They conjecture that this is due to the implicit regularization that the gradient descent tends to find a very smooth function among all global minimizers. Moreover, we typically need to forecast the future behavior of the same dynamics or predict other dynamics, which inevitably leads to the generalization error. Whereas the numerical results have shown excellent generalization performance \cite{du2022discovery, raissi2018multistep}, the error out of sample locations needs further precise quantification.

We partially solved these open questions, i.e., quantifying the error out of sample locations without relying on the auxiliary conditions. The main ingredients are the formal analysis \cite{feng1991formal,feng1993formal} as well as the inverse modified differential equations (IMDE) proposed in our conference paper \cite{zhu2022on}. Following the framework introduced in \cite{keller2021discovery}, we first find a perturbed differential equation, i.e., the IMDE, which satisfies the multistep relation (\ref{eq: relation of lmnets}) formally. The explicit recursion formula of IMDE is given by means of the Lie derivatives and the multistep formula. In addition, inspired by the rigorous estimates of modified differential equations \cite{benettin1994hamiltonian, hairer1997life,hairer1999backward, reich1999backward}, we truncate IMDE and prove that the difference between the learned network and the truncated IMDE is bounded by the learning loss and a discrepancy that can be made exponentially small, which implies the uniqueness of the solution of the learning task (\ref{eq: loss of lmnets}) in some sense. Via investigating IMDE, we show that the error between the learned network and the target governing function is bounded by the sum of the LMM discretization error $\mathcal{O}(h^p)$ and the learning loss when using a $p$-th order LMM with step size $h$. Subsequently, we combine the approximation and generalization theories of neural networks to establish a priori error estimates of LMNets by quantifying the learning loss with the width of the employed network and the number of samples. It is noted that our conclusions do not depend on auxiliary conditions, and hold for any weakly stable and consistent LMM.

This paper is organized as follows: In \cref{sec: pre}, we briefly present some necessary notations and preliminaries. In \cref{sec: imde}, we give the explicit recursion formula of the IMDE of LMMs. In \cref{sec: error}, we study the difference between the target governing function and the suitably truncated IMDE, and present the main results and their detailed proofs. In \cref{sec: numerical results}, we provide several numerical results to verify the theoretical findings. Finally, we summarize our work in \cref{sec: summary}.

\section{Preliminaries}\label{sec: pre}
\subsection{Notations}
In this subsection, we collect together some of the notations and definitions introduced throughout the paper.
\begin{enumerate}
\item In this paper, we will work with the $\ell_1$ and $\ell_2$-norm on $\C^D$, as well as their corresponding tensor norms, which we denote as $\norm{\cdot}_1$, $\normt{\cdot}$, respectively. In addition, unless otherwise specified, we omit the subscript of $\ell_1$-norm for convenience, $\norm{\cdot }:=\norm{\cdot}_1$.
\item For a function $g$ defined on $\Omega$, we denote its sup-norm on $\Omega$ by 
\begin{equation*}
\norm{g}_{\Omega} = \sup_{x\in\Omega}\norm{g(x)}_1.
\end{equation*}
In addition, if $g$ is uniformly Lipschitz function, we denote its Lipschitz constant by
\begin{equation*}
\lip{g}_{\Omega} := \sup_{x, y \in \Omega. x\neq y} \frac{\normt{g(x) - g(y)}}{\normt{x-y}}.
\end{equation*} 
\item Let $\mathcal{B}(x,r) := \{y\in \R^D, \norm{y-x}_1\leq r\} \subset \R^D$ be the ball of radius $r>0$ centered at $x\in\R^D$. For a subset $\Omega \subset \R^D$, denote its neighborhood by $\mathcal{B}(\Omega, r) = \bigcup_{x \in \Omega} \mathcal{B}(x,r)$.
\item Let $\mathcal{CB}(x,r) := \{y\in \C^D, \norm{y-x}_1\leq r\} \subset \C^D$ be the complex ball of radius $r>0$ centered at $x\in\C^D$. For a compact subset $\Omega \subset \C^D$, denote its complex neighborhood by $\mathcal{CB}(\Omega, r) = \bigcup_{x \in \Omega} \mathcal{CB}(x,r)$. 
\end{enumerate}
\begin{definition}
A function $g: \Omega \rightarrow \C^{D'}$ is called $(Q, r)$-analytic on $\Omega \subset \R^{D}$ if $g$ is analytic on $\mathcal{CB}(\Omega, r)$ and $\norm{g}_{\mathcal{CB}(\Omega, r)} \leq Q$.
\end{definition}

\begin{definition}
A function $g: \Omega \rightarrow \R^{D'}$ is called $(Q, r, I)$-analytic on $\Omega \subset \R^{D}$ if $\|g^{(i)}\|_{\Omega} \leq i! \cdot Q \cdot r^{-i}$, for $i=0,\cdots, I$.
\end{definition}
It is noted that if $g$ is $(Q, r)$-analytic on $\Omega$, then, $g$ is $(Q, r, I)$-analytic on $\Omega$ for $\forall I\in \N$. This property can be found in Section 4.1 in \cite{hairer1997life}, and is the reason we work with the $\ell_1$-norm.

\subsection{Dynamical systems}
Without loss of generality, in this paper, we focus on autonomous systems of first-order ordinary differential equations
\begin{equation}\label{eq:ODE}
\frac{d}{dt}y(t) = f(y(t)),\quad y(0)=x,
\end{equation}
where $y(t) \in \R^D$ and $f:\R^D \rightarrow \R^D$ is smooth. The initial value is denoted by $x$ in this paper. A non-autonomous system $\frac{d}{dt}y(t) = f(t,y(t),p)$ with parameter $p$ can be brought into this form by appending the equations $\frac{d}{dt}t = 1$ and $\frac{d}{dt}p = 0$. For fixed $t$, $y(t)$ can be regarded as a function of its initial value $x$. We denote
\begin{equation*}
\phi_{t}(x):=y(t) = x + \int_0^t f(y(\tau))d\tau,
\end{equation*}
which is known as the time-$t$ flow map of dynamical system (\ref{eq:ODE}). In the following sections, we will add the subscript $f$ to specify a particular differential equation, denoting $\phi_t$ as $\phi_{t,f}$.

\subsection{Linear multistep methods}
For first order differential equations (\ref{eq:ODE}), linear multistep methods (LMMs) are defined by the formula
\begin{equation}\label{eq:LMM}
\sum_{m=0}^M \alpha_my_{m} - h\sum_{m=0}^M\beta_mf(y_{m})=0,
\end{equation}
where $\alpha_m,\beta_m$ are specified coefficients that satisfy $\alpha_M \neq 0$ and $|\alpha_0|+ |\beta_0|>0$. When applying an LMM to solve a differential equation (\ref{eq:ODE}), a starting method for initial $M$ values $y_1,\cdots, y_{M-1}$ must be chosen and the approximations $y_n$ for $n \geq M$ can then be computed recursively. The method (\ref{eq:LMM}) is called explicit if $\beta_M=0$, otherwise it is implicit and $y_n$ for $n \geq M$ can be computed iteratively by fixed-point iteration or Newton-Raphson iteration. Next, we present some basic notations and concepts following \cite{hairer2006geometric,hairer1993solving}.\\
\textbf{Weak stability.}
An LMM is weakly stable if \begin{equation}\label{eq:weak stability}
\sum_{m=0}^M m\cdot \alpha_m \neq 0.
\end{equation}
\textbf{Consistency.}
An LMM is consistent if
\begin{equation}\label{eq:consistency}
\sum_{m=0}^M \alpha_m =0,\ \text{and }\sum_{m=0}^M m\cdot \alpha_m =\sum_{m=0}^M\beta_m.
\end{equation}
\textbf{Order.}
An LMM is of order $p$ ($p \geq 1$) if
\begin{equation}\label{eq:order}
\sum_{m=0}^M \alpha_m =0,\ \text{and }\sum_{m=0}^M \frac{m^{k+1}}{(k+1)!}\cdot \alpha_m =\sum_{m=0}^M\frac{m^{k}}{k!} \cdot \beta_m\ \text{for }k=0,1,\cdots,p-1.
\end{equation}
In this paper, we concentrate on weakly stable and consistent linear multistep methods, whose coefficients satisfy $\sum_{m=0}^M\beta_m \neq 0$ due to the weak stability condition (\ref{eq:weak stability}) and the consistency condition (\ref{eq:consistency}). Thus we assume $\sum_{m=0}^M\beta_m =1$ without loss of generality. 

\subsection{Neural networks}
We briefly describe the setting of the supervised learning problem using neural networks in this subsection. Let $\M$ denote the hypothesis class consisting of neural networks. In this paper, we will focus on a widely used architecture known as feedforward neural networks (FNNs). Mathematically speaking, an FNN, denoted by $f_{\theta}(x)$, is composed of a series of linear layers and activation layers, i.e.,
\begin{equation*}
\begin{aligned}
f_{\theta}(x) = W_{L+1} h_L\circ h_{L-1}\circ \cdots \circ h_1(x) +b_{L+1},
\end{aligned}
\end{equation*}
where
\begin{equation*}
h_l(z_l) =\sigma (W_lz_l+b_l),\ W_l\in \R^{d_l \times d_{l-1}}, b_l\in \R^{d_l} \text{ for }l=1,\cdots,L+1.
\end{equation*}
The number $L$ is the total number of layers, $\theta =\{W_l, b_l\}_{l=1}^{L+1}$ are trainable parameters, and $\sigma:\R \rightarrow \R$ is a predetermined activation function that is applied element-wise to a vector. Popular examples include the rectified linear unit (ReLU) $\texttt{ReLU}(z) = \max(0, z)$, the sigmoid $\texttt{Sig}(z) = 1/(1 + e^{-z})$ and $\texttt{tanh}(z)$. In addition, we denote the maximum value of elements in $\theta$ by 
\begin{equation*}
\max \theta = \max\{W_l^{ij}, b_l^i\}_{l=1,\cdots,L+1} \text { for } \theta = \{W_l=[W_l^{ij}], b_l = [b_l^i]\}_{l=1,\cdots,L+1}.
\end{equation*}

Let $x\in \mathcal{X}$ be inputs, $f(x)$ be the corresponding targets, and $\ell(\cdot,\cdot)$ be a loss function\footnote{A common choice for regression problems is the square loss, $\ell(y, \hat{y}) = \norm{y-\hat{y}}_2^2$.}. The goal of supervised learning is typically framed as an optimization problem of the form 
\begin{equation*}
\inf_{f_{\theta} \in \M} \mathcal{L}(f_{\theta}), \ \text{where }\mathcal{L}(f_{\theta}) = \int_{\mathcal{X}} \ell(f_{\theta}(x), f(x)) dP(x).
\end{equation*}
The objective function $\mathcal{L}(f_{\theta})$ is known as the expected loss of a network $f_{\theta}$. In practice, the probability measure $P$ is unknown. And thus we sample training data $\{(x_n,f(x_n))\}_{n=1}^N$ and set $P$ to be the empirical measure, yielding the empirical risk optimization problem
\begin{equation}\label{eq:training loss}
\inf_{f_{\theta} \in \M} \mathcal{L}_N(f_{\theta}), \ \text{where }\mathcal{L}_N(f_{\theta})=\frac{1}{N}\sum_{n=1}^N \ell(f_{\theta}(x_n), f(x_n)).
\end{equation}
The objective function $\mathcal{L}_N(f_{\theta})$ is referred to as the training loss since it is the loss function evaluated the training data.

\subsubsection{Approximation property} It is well-known that FNNs can approximate essentially any continuous function if its size is sufficiently large \cite{cybenko1989approximation, devore2021neural,hornik1990universal,lu2021deep}. Next, we present a recently developed result regarding the approximation property of tanh FNNs.
\begin{theorem}\label{thm:approximation}
Let $D\in \N$, $Q,r,\delta>0$, $\Omega\subset \R^D$ open with $[0,1]^D \subset \Omega$ and let $f:\Omega\rightarrow \R$ be $(Q, r)$-analytic on $\Omega$. Then for every $s \geq 1$ and $\widehat{N}>3D/2$, there is a tanh FNN $f_{\widehat{\theta}}$ with two hidden layers of widths at most $3\left\lceil\frac{s}{2}\right\rceil\binom{s+D-1}{s-1}+D(\widehat{N}-1)$ and $3\left\lceil\frac{D+2}{2}\right\rceil\binom{2D+1}{D+1}\widehat{N}^D$ such that
\begin{equation*}
\norm{f-f_{\widehat{\theta}}}_{[0,1]^D} \leq (1+\delta) Q \left(\frac{3D}{2r\widehat{N}}\right)^{s},
\end{equation*} 
and for $k=1, \ldots, s-1$,
\begin{equation*}
\norm{f-f_{\widehat{\theta}}}_{W^{k, \infty}([0,1]^D)} \leq \alpha \max\left\{r^k,\ln^k\left(\beta \widehat{N}^{s+D+2}\right)\right\} \cdot \frac{\mathbf{C}(D,k,s,f)}{\widehat{N}^{s-k}},
\end{equation*}
where
\begin{equation*}
\begin{aligned}
&\alpha = 3^D \left(1+\delta\right) (2(k+1))^{3k},\quad \beta = \frac{k^3 2^{D}\sqrt{D}\max\{1,\norm{f}^{1/2}_{W^{k, \infty}([0,1]^D)} \}}{\delta\min\{1,\sqrt{\mathbf{C}(D,k,s,f)}\} },\\
&\mathbf{C}(D,k,s,f) = \max_{0\leq k' \leq k} \frac{1}{(s-k')!}\left(\frac{3D}{2}\right)^{s-k'}\frac{s!Q}{r^s}. 
\end{aligned}
\end{equation*}
In addition, the weights of $f_{\widehat{\theta}}$ scale as $\mathcal{O}\brac{\mathbf{C}^{-s/2}\widehat{N}^{D(D+s^2+k^2)/2}\cdot (s(s+2))^{3s(s+2)}}$.
\end{theorem}
\cref{thm:approximation} is a direct result of Theorem 5.1 and Corollary 5.5 in \cite{tim2021on}. There are a variety of available results on the approximation of the ReLU FNNs \cite{devore2021neural, lu2021deep} and the Floor ReLU FNNs \cite{shen2021deep}. While some of them can lessen the curse of dimensionality, we only present the aforementioned approximation results of the tanh FNNs due to the requirements for smoothness below. We would like to combine other advanced approximation results with the error analysis of the discovery of dynamics in future works.

\subsubsection{Generalization property}
Here we introduce existing results on the generalization properties of FNNs \cite{shin2020convergence}, which provide an upper bound on the expected loss. To this end, we introduce the following assumptions about the distribution of the training data:
\begin{assumption}\label{assumption:data-dist}
Let $\Omega$ be bounded domain in $\R^D$ that is at least of class $C^{0,1}$, let $\mu_{\Omega}$ be probability distributions defined on $\Omega$. Let $\rho_{\Omega}$ be the probability density of $\mu_{\Omega}$ with respect to $D$-dimensional Lebesgue measure on $\Omega$.
\begin{enumerate}
\item $\rho_{\Omega}$ is supported on $\overline{\Omega}$ and $\inf_{\Omega} \rho_{\Omega}:=c_{inf} > 0$.
\item
For $\epsilon > 0$, there exists partitions of $\Omega$, $\{\Omega_{j}^\epsilon\}_{j=1}^{J}$ that depend on $\epsilon$ such that for each $j$, there are cubes $H_{\epsilon}(z_{j})$ of side length $\epsilon$ centered at $z_{j} \in \Omega_{j}^\epsilon$, satisfying $\Omega_{j}^\epsilon \subset H_{\epsilon} (\mathbf{z}_{j})$.
\item
There exists positive constant $c_{s}$ such that $\forall \epsilon > 0$, the partitions from the above satisfy $c_{s} \epsilon^{D} \le \mu_{\Omega}(\Omega_{j}^\epsilon)$ for all $j$.
There exists positive constant $C_{s}$ such that for $\forall x \in \Omega$, we have $\mu_{\Omega} (B_{\epsilon}(\mathbf{x}) \cap \Omega) \le C_{s}\epsilon^D$, where $B_\epsilon(x)$ is a closed ball of radius $\epsilon$ centered at $x$.
Here $C_{s}, c_{s}$ depend only on $(\Omega,\mu_{\Omega})$.
\end{enumerate}
\end{assumption}
Now, the generalization results of FNNs is given as follows. \cite{shin2020convergence}.
\begin{theorem} \label{thm:gen}
Suppose \cref{assumption:data-dist} holds. Let $N$ be the number of iid samples from $\mu_{\Omega}$. Suppose operator $\mathcal{G}$, function $u$, and neural networks $f_{\theta}$ satisfy $\lip{u}_{\Omega}< \infty,\quad \lip{\mathcal{G}(f_{\theta})}_{\Omega} < \infty$, 
Then, with probability at least $(1 - \sqrt{N}(1-1/\sqrt{N})^{N})$,
we have
\begin{equation*}
\small
\int_{\Omega} \normt{\mathcal{G}(f_{\theta})(x) - u(x)}^2 dx \leq 
 \frac{C_{N}}{N} \sum_{n=1}^N \normt{\mathcal{G}(f_{\theta})(x_n) - u(x_n)}^2 + 
\brac{\hat{\lambda}_{N}\lip{\mathcal{G}(f_{\theta})}_{\Omega}^2+C'}N^{-\frac{1}{D}},
\end{equation*}
where $C_{N} = 3 D^{D/2}C_s/c_s \cdot N^{1/2}/c_{inf}$, $\hat{\lambda}_{N} = 3 D c_s^{-2/D} $ and
$C'$ is a universal constant that depends only on $D$, $c_s$, $u$.
\end{theorem}
It is remarked that the generalization properties of neural network models are often better in practice than in theory, and we can measure their learning performance a posteriori by evaluating the test loss on  test data.

\subsubsection{Regularization}
We introduce regularization in this subsection. For learning problem (\ref{eq:training loss}), we jointly minimize the empirical risk and a regularization function $\mathcal{R}$, resulting in the following optimization problem:
\begin{equation*}
\inf_{f_{\theta} \in \M} \mathcal{L}_N(f_{\theta}) + \mathcal{R}(\theta).
\end{equation*}
Examples of regularization function include Tikhonov regularization \cite{shalev2014understanding} (i.e., $\ell_p$ norm of $\theta$) and H\"older regularization \cite{shin2020convergence} (i.e., norm of the derivative of $f_{\theta}$).

The use of over-parameterized FNNs without explicit regularization typically leads to the non-uniqueness of global minimizers, but there is further implicit regularization to prevent over-fitting. It has been shown that training neural networks using gradient descent learns the low-frequency modes first and then fits the high-frequency modes to eliminate the training loss \cite{cao2021towards,rahaman2019on,xu2022overview,xu2019training}. This implicit bias indicates that networks trained by gradient descent prioritize learning smooth functions with low frequencies among all solutions with near-zero training loss. Such implicit regularization was first discussed in \cite{du2022discovery} for the discovery of dynamics, and also motivates us to make the analyticity assumptions in this work, which will be documented in detail later.

\subsection{Discovery of dynamics using LMMs}
The optimization problem (\ref{eq: loss of lmnets}) is specific to the discovery on single trajectory. More generally, we consider training data given as
\begin{equation*}
\mathcal{T}_{train} = \left\{\big(x_n, \phi_{h,f}(x_n), \cdots, \phi_{Mh,f}(x_n)\big)\right\}_{n=1}^N.
\end{equation*}
Then the training loss (\ref{eq: loss of lmnets}) of the dynamics discovery can be rewritten as
\begin{equation}\label{eq:loss-of-lmnets}
\mathcal{L}_{h,N}(f_{\theta})= \frac{1}{N}\sum_{n=1}^{N}\norm{\sum_{m=0}^M h^{-1} \alpha_m\phi_{mh,f}(x_n) - \sum_{m=0}^M\beta_mf_{\theta}(\phi_{mh,f}(x_n))}_2^2,\ f_{\theta} \in \M,
\end{equation}
which is formally close to (\ref{eq:training loss}) except for the specially designed loss function. If the states at equidistant time steps of some trajectories are given, we can group them in pairs to get the above form. {Specifically, suppose we are given a dataset $\{\phi_{n_1h,f}(\tilde{x}_{n_2})\}_{n_1=0, \cdots, N_1;\ n_2=1,\cdots,N_2}$ with initial points set $\{\tilde{x}_{n_2}\}_{ n_2=1,\cdots,N_2}$, we can denote them as 
\begin{equation*}
\left\{\big(\phi_{n_1h,f}(\tilde{x}_{n_2}), \phi_{(n_1+1)h,f}(\tilde{x}_{n_2}), \cdots, \phi_{(n_1+M)h,f}(\tilde{x}_{n_2})\big)\right\}_{n_1=0,\cdots, N_1-M;\ n_2=1,\cdots,N_2}.
\end{equation*}}

Similar to the standard supervised learning problem, the gradient-based optimization techniques in deep learning are able to not only minimize the training loss, but also achieve small loss on unknown data. After training, the returned network $f_{\theta}$ serves as an approximation of the target governing function $f$ in a closed-form.

\section{Inverse modified differential equations of LMMs}\label{sec: imde}

To begin with, we seek a perturbed differential equation
\begin{equation}\label{eq:imde}
\begin{aligned}
\frac{d}{dt}\tilde{y}(t)=&f_h(\tilde{y}(t))=f_0(\tilde{y})+hf_1(\tilde{y})+h^2f_2(\tilde{y})+\cdots,
\end{aligned}
\end{equation}
such that formally
\begin{equation}\label{eq:def imde}
\sum_{m=0}^M \alpha_m\phi_{mh,f}(x) = h\sum_{m=0}^M\beta_mf_h(\phi_{mh,f}(x)),
\end{equation}
where the identity is understood in the sense of the formal power series in $h$ without taking care of convergence issues. In this paper, we name the perturbed equation (\ref{eq:imde}) as the inverse modified differential equation (IMDE), since modified differential equation \cite{feng1991formal, feng1993formal, hairer1999backward, hairer2006geometric} for forward problem is also a perturbed differential equation $\frac{d}{dt}\hat{y}(t)=\hat{f}_h(\hat{y}(t))$ that satisfies an analogous relationship:
\begin{equation*}
\sum_{m=0}^M \alpha_m\phi_{mh,\hat{f}_h}(x) = h\sum_{m=0}^M\beta_mf(\phi_{mh,\hat{f}_h}(x)).
\end{equation*}

To this end, we first briefly review Lie derivatives following \cite{hairer2006geometric}. Given dynamical system (\ref{eq:ODE}), Lie derivative $\D$ is the differential operator defined as:
\begin{equation*}
\D g(y)=g'(y)f(y), \text{ for }g:\R^D \rightarrow \R^D. 
\end{equation*}
Therefore, using the chain rule we derive that
\begin{equation*}
\frac{d}{dt}g(\phi_{t,f}(x))=(\D g)(\phi_{t,f}(x))
\end{equation*}
and thus obtain the Taylor series of $g(\phi_{t,f}(x))$ developed at $t=0$:
\begin{equation}\label{eq:Ld}
g(\phi_{t,f}(x))=\sum_{k=0}^{\infty}\frac{t^k}{k!}(\D^kg)(x).
\end{equation}
In particular, by setting $g(y)=I_D(y)=y$, the identity map, we obtain the Taylor series of the exact solution $\phi_{t,f}$ itself, i.e.,
\begin{equation}\label{eq:exasolu}
\begin{aligned}
\phi_{t,f}(x)&=\sum_{k=0}^{\infty}\frac{t^k}{k!}(\D^kI_D)(x)\\
&=x+tf(x)+\frac{t^2}{2}f'f(x)+\frac{t^3}{6}(f''(f,f)(x)+f'f'f(x))+\cdots .
\end{aligned}
\end{equation}
Here, the notation $f'(x)$ is a linear map (the Jacobian), the second order derivative $f''(x)$ is a symmetric bilinear map and similarly for higher order derivatives described as tensors.

Now we are able to present the general formula of the IMDE (\ref{eq:imde}) in a recursive manner.
\begin{theorem}\label{the:inmde_lmnet}
Consider the dynamical system (\ref{eq:ODE}) and a weakly stable and consistent LMM (\ref{eq:LMM}), there exist unique h-independent functions $f_k$ such that for any truncation index $K$, the truncated IMDE $f_h^K =\sum_{k=0}^K h^kf_k$ satisfies
\begin{equation}\label{eq:multistep}
\sum_{m=0}^M \alpha_m\phi_{mh,f}(x) = h\sum_{m=0}^M\beta_mf_h^K(\phi_{mh,f}(x)) +\mathcal{O}(h^{K+2}).
\end{equation}
In particular, for $k \geq 0$, the functions $f_{k}$ are given as
\begin{equation}\label{eq:definition of f_k}
\begin{aligned}
f_k(y)=& \sum_{m=0}^M \alpha_m \frac{m^{k+1}}{(k+1)!}(\D^kf)(y)-\sum_{m=0}^M\beta_m\sum_{j=1}^k\frac{m^j}{j!}(\D^jf_{k-j})(y).
\end{aligned}
\end{equation}
\end{theorem}

\begin{proof}
The approach for the computation of $f_h$ is presented in two steps. To begin with, by using the formula (\ref{eq:exasolu}), the left of (\ref{eq:multistep}) can be expanded as
\begin{equation}\label{eq:msleft}
\begin{aligned}
\sum_{m=0}^M \alpha_m\phi_{mh,f}(x) &= \sum_{m=0}^M \alpha_m \sum_{k=0}^{\infty}\frac{(mh)^k}{k!}(\D^kI_D)(x)\\
&=\sum_{k=0}^{\infty}h^k \left[\sum_{m=0}^M \alpha_m \frac{m^k}{k!}(\D^k I_D)(x)\right].
\end{aligned}
\end{equation}
Subsequently, using (\ref{eq:Ld}) with setting $t=mh$ and $g(y)=f_h(y)$, we obtain that
\begin{equation*}
\begin{aligned}
h\sum_{m=0}^M\beta_mf_h(\phi_{mh,f}(x))
=&h\sum_{m=0}^M\beta_m \sum_{j=0}^{\infty}\frac{(mh)^j}{j!} \sum_{i=0}^{\infty}h^i(\D^jf_i)(x).
\end{aligned}
\end{equation*}
By interchanging the summation order, we deduce that
\begin{equation}\label{eq:msright}
\begin{aligned}
h\sum_{m=0}^M\beta_mf_h(\phi_{mh,f}(x))
=&h\sum_{m=0}^M\beta_m \sum_{k=0}^{\infty}h^k \sum_{j=0}^k\frac{m^j}{j!}(\D^jf_{k-j})(x) \\
=&\sum_{k=0}^{\infty}h^{k+1} \sum_{m=0}^M\beta_m\left[f_k(x)+\sum_{j=1}^k\frac{m^j}{j!}(\D^jf_{k-j})(x)\right].
\end{aligned}
\end{equation}
Comparing coefficients of $h^k$ in (\ref{eq:msleft}) and (\ref{eq:msright}) for $k=0,1,2,\cdots$ yields
\begin{equation*}
\sum_{m=0}^M\alpha_m=0,
\end{equation*}
the consistency condition, and
\begin{equation*}
\begin{aligned}
\sum_{m=0}^M\beta_m\left[f_k(x)+\sum_{j=1}^k\frac{m^j}{j!}(\D^jf_{k-j})(x)\right]
=\sum_{m=0}^M \alpha_m \frac{m^{k+1}}{(k+1)!}(\D^{k+1}I_D)(x).
\end{aligned}
\end{equation*}
Therefore, we conclude that
\begin{equation*}
\begin{aligned}
f_k(y)=&\sum_{m=0}^M \alpha_m\frac{m^{k+1}}{(k+1)!}(\D^kf)(y)-\sum_{m=0}^M\beta_m\sum_{j=1}^k\frac{m^j}{j!}(\D^jf_{k-j})(y),
\end{aligned}
\end{equation*}
which uniquely defines the functions $f_k$ in a recursive manner. Here, we have used the fact that $\sum_{m=0}^M \beta_m =1$ and $\D^{k+1}I_D = \D^kf$.
\end{proof}

As a direct consequence of the order condition (\ref{eq:order}) and the recursion formula (\ref{eq:definition of f_k}), we have the following corollary.
\begin{corollary}\label{cor:fh-f}
Under notations and conditions of \cref{the:inmde_lmnet}, if the LMM is of order $p$, then, the IMDE obeys
\begin{equation*}
\frac{d}{dt}\tilde{y}=f_h(\tilde{y})=f(\tilde{y})+h^pf_p(\tilde{y})+\cdots,
\end{equation*}
where
\begin{equation*}
\begin{aligned}
f_p(y)=\left[\sum_{m=0}^M \frac{m^{p+1}}{(p+1)!} \cdot \alpha_m-\sum_{m=0}^M\frac{m^p}{p!}\cdot \beta_m\right](\D^pf)(y).
\end{aligned}
\end{equation*}
\end{corollary}

\section{Error analysis}\label{sec: error}

With the correspondence established between the definition of IMDE (\ref{eq:def imde}) and the LMM loss (\ref{eq:loss-of-lmnets}), we are now ready to show that training LMNets returns a close approximation of the truncated IMDE. This leads us to formulate the error representation results. 

In this section, we first present the main theorems in \cref{sec:Main results}, followed by a discussion on the necessity of the analyticity assumptions in \cref{sec:The necessity of the analyticity assumptions}. All the proofs can be found in \cref{sec:proofs}. Throughout this section, we assume that the employed LMM is weakly stable, consistent, and $\sum_{m=0}^M\beta_m =1$.

\subsection{Main results}\label{sec:Main results}
The main theorem is given as follows, which indicates that the difference between the learned network and the truncated IMDE is bounded by the learning loss and a discrepancy that can be made exponentially small. 

{
\begin{theorem}\label{the:error}
Consider the $D$-dimensional dynamical system (\ref{eq:ODE}) and a weakly stable and consistent LMM (\ref{eq:LMM}) of order $p$ with coefficients $\alpha_m$, $\beta_m$. Let $ Q,\ R,\ r>0$ and $\K\subset \R^D$ be a bounded domain. Let $f_{\theta^*}$ be the network learned by optimizing (\ref{eq:loss-of-lmnets}) and denote 
\begin{equation*}
\mathcal{L} =\int_{\mathcal{B}(\K, R)} \norm{\sum_{m=0}^M h^{-1}\alpha_m\phi_{mh,f}(x) - \sum_{m=0}^M\beta_mf_{\theta^*}(\phi_{mh,f}(x))}_2^2 dx
\end{equation*}
Suppose the target governing function $f$ is $(Q, r)$-analytic on $\mathcal{B}(\K, R)$ and the learned vector field $f_{\theta^*}$ is $(Q, r, I)$-analytic on $\mathcal{B}(\K, R)$. Then, there exists a uniquely defined vector field, i.e., the truncated IMDE $f_h^K$, such that, if $0<h<h_0$,
\begin{equation*}
\begin{aligned}
&\int_{\K}\norm{f_{\theta^*}(x) - f_h^K(x)} dx \leq c_1 e^{-\gamma/h^{2/5}} +c_0 \mathcal{L}^{1/2},\\
&\int_{\K}\norm{f_{\theta^*}(x) - f(x)} dx \leq c_2 h^p + c_0 \mathcal{L}^{1/2},
\end{aligned}
\end{equation*}
where constants $\gamma$, $h_0$, $c_0$, $c_1$, $c_2$ depend on $r/Q$, $R/Q$, the dimension $D$, volume of the domain $\mathcal{B}(\K, R)$ and the coefficients $\alpha_m$ and $\beta_m$, and integers $K=K(h)$, $I=I(h)$.
\end{theorem}
}
We refer to $\mathcal{L}$ as the learning loss, as it represents the loss function evaluated over the neighborhood of the given domain, and provides an indication of the learning performance of the trained model. The IMDE is inaccessible in practice due to the unknow of the target governing function. However, it implies the uniqueness of the solution of the learning task (\ref{eq:loss-of-lmnets}) in some sense, and allows us to formulate the error estimation. If $\mathcal{L}$ is sufficiently small, $\int_{\K}\norm{f_{\theta^*}(x) - f_h^K(x)} dx$ is exponentially small and the convergence rate LMNets with respect to $h$ is consistent with the order of the LMM.

{
\cref{the:error} can be applied to any type of approximation structure. Specifically, we combine the approximation and generalization theories of neural networks to quantify the learning loss by taking into account the width of the employed network and the number of samples, yielding the following error estimate:
\begin{theorem}\label{the:error2}
Consider the $D$-dimensional dynamical system (\ref{eq:ODE}) and a weakly stable and consistent LMM (\ref{eq:LMM}) of order $p$ with coefficients $\alpha_m$, $\beta_m$. Let $\delta,\ Q,\ R,\ r>0$, $N,\ s,\ \widehat{N} \in \N$ with $s\geq 1$, $\widehat{N}\geq 3D/2$, $N^{1/(2D)+1/4} \leq e^{ \gamma_0/h} / (\rho Q)$. Let $\K\subset \R^D$ be a bounded domain. Let $f_{\theta^*}$ be the global minimizer of $\mathcal{L}_{h,N}$ (\ref{eq:loss-of-lmnets}), where the hypothesis class $\M$ consists of all tanh FNNs with two hidden layers of widths $3D\left\lceil\frac{s}{2}\right\rceil\binom{s+D-1}{s-1}+D^2(\widehat{N}-1)$ and $3D\left\lceil\frac{D+2}{2}\right\rceil\binom{2D+1}{D+1}\widehat{N}^D$. Suppose that the following assumptions hold:
\begin{itemize}
 \item The target governing function $f$ is $(Q, r)$-analytic on $\mathcal{C}$,
 where $\mathcal{C} \subset \R^D$ is a closed cuboid satisfying $\mathcal{B}(\K, R) \subset \mathcal{C}$.
 \item \cref{assumption:data-dist} holds with $\Omega = \mathcal{B}(\K, R)$, where the number of iid samples from $\mu_{\mathcal{B}(\K, R)}$ is $N$.
 \item $f_{\theta^*}$ is $(Q, r, I)$-analytic on $\mathcal{B}(\K, R)$.
\end{itemize}
If $0<h<h_0$, there exists a uniquely defined vector field, i.e., the truncated IMDE $f_h^K$, such that, with probability at least $(1 - \sqrt{N}(1-1/\sqrt{N})^{N})$, we have
\begin{equation*}
\begin{aligned}
&\int_{\K}\norm{f_{\theta^*}(x) - f_h^K(x)} dx \leq C_1 e^{-\gamma/h^{2/5}} +C_{0,1}N^{1/4}\left(\frac{C_{0,2}}{\widehat{N}}\right)^{s} + C_{0,3}N^{-\frac{1}{2D}},\\
&\int_{\K}\norm{f_{\theta^*}(x) - f(x)} dx \leq C_2 h^p +C_{0,1}N^{1/4}\left(\frac{C_{0,2}}{\widehat{N}}\right)^{s} + C_{0,3}N^{-\frac{1}{2D}},
\end{aligned}
\end{equation*}
where $\gamma_0$, $\rho$, $\gamma$, $h_0$, $C_{0,1}$, $C_{0,2}$, $C_{0,3}$, $C_1$ and $C_2$ are constants independent of $h$, $N$, $s$, $\widehat{N}$, integers $K=K(h)$, $I=I(h)$.
\end{theorem}

If there are enough data points ($N^{1/(2D)+1/4} = \mathcal{O}(e^{ \gamma_0/h})$) and the network is sufficiently large, the difference between the learned network and the truncated IMDE is exponentially small, and the error between the learned network and the target governing function will decay to zero with the rate $\mathcal{O}(h^p)$ when using a $p$-th order LMM.

The analyticity assumptions in \cref{the:error} and \cref{the:error2} are used to measure the smoothness of the target governing function and the learned vector field $f_{\theta}$, and serve as the fundamental basis for our proofs. As discussed in \cite{du2022discovery}, it is natural to conjecture that the implicit regularization that neural networks fit targets from low to high frequencies during gradient descent can be applied to LMNets. If the target system is analytic, the gradient descent is expected to converge to an analytic solution without high-frequency components. We will numerically verify this fact in \cref{sec: numerical results}. 

Subsequently, we show that the $(Q, r, I)$-analyticity assumption of  the learned vector field $f_{\theta}$ holds under the designed explicit regularization. With \cref{assumption:data-dist}, for any $ \varepsilon, r_1>0$ and $I \in \N$, we define the following regularization function:
\begin{equation*}
\mathcal{R}_{\varepsilon, r_1, I}(\theta) = \sum_{i=0}^{I} \frac{r_1^i}{i!}\max\left\{\norm{f_{\theta}^{(i)}(z_j)} \right\}_{j=1}^J + \frac{\max \theta}{\mathbf{C}^{-s/2}\widehat{N}^{D(D+s^2+k^2)/2}(s(s+2))^{3s(s+2)}},
\end{equation*}
where points $z_j$ corresponding to $\varepsilon$ are given as the second term of \cref{assumption:data-dist}. Next, we define the regularized loss as follows:
\begin{equation}\label{eq:regularized training loss}
\mathcal{L}_{h,N}^{\mathcal{R}} (f_{\theta}) = \mathcal{L}_{h,N} (f_{\theta}) + N^{-1/D-1/2}\mathcal{R}_{\varepsilon, r_1, I}(\theta), \ f_{\theta} \in \M.
\end{equation}
Then, the error estimate for optimizing the regularized loss is given as follows:
\begin{theorem}\label{the:error3}
Under the notations and the first two assumptions of \cref{the:error2}. Let $f_{\theta^*}^{\mathcal{R}}$ be the global minimizer of $\mathcal{L}_{h,N}^{\mathcal{R}}$ (\ref{eq:regularized training loss}), where the hypothesis space $\M$ is consistent with that in \cref{the:error2}.

If $0<h<h_0$, $N^{1/(2D)+1/4} \leq e^{ \gamma_0/h} / (\rho Q)$, $I\geq \gamma_1/h^{2/5}$, $s\geq I+1$, $\widehat{N} \geq \widehat{N}_0(N)$ and $0<\varepsilon \leq \varepsilon_0(s, \widehat{N})$, $ r_1\leq r/2$, then, there exists a uniquely defined vector field, i.e., the truncated IMDE $f_h^K$, such that, with probability at least $(1 - \sqrt{N}(1-1/\sqrt{N})^{N})$, we have
\begin{equation*}
\begin{aligned}
&\int_{\K}\norm{f_{\theta^*}^{\mathcal{R}}(x) - f_h^K(x)} dx \leq C_1^{\mathcal{R}} e^{-\gamma/h^{2/5}} +C_{0,1}^{\mathcal{R}} N^{1/4}\left(\frac{C_{0,2}^{\mathcal{R}}}{\widehat{N}}\right)^{s} + C_{0,3}^{\mathcal{R}} N^{-\frac{1}{2D}},\\
&\int_{\K}\norm{f_{\theta^*}^{\mathcal{R}}(x) - f(x)} dx \leq C_2^{\mathcal{R}} h^p +C_{0,1}^{\mathcal{R}} N^{1/4}\left(\frac{C_{0,2}^{\mathcal{R}}}{\widehat{N}}\right)^{s} + C_{0,3}^{\mathcal{R}} N^{-\frac{1}{2D}},
\end{aligned}
\end{equation*}
where $\gamma_0$, $\rho$, $\gamma_1$, $\gamma$, $h_0$, $C_{0,1}^{\mathcal{R}}$, $C_{0,2}^{\mathcal{R}}$, $C_{0,3}^{\mathcal{R}}$, $C_1^{\mathcal{R}}$ and $C_2^{\mathcal{R}}$ are constants independent of $h$, $N$, $s$, $\widehat{N}$.
\end{theorem}
}

\subsection{The necessity of the analyticity assumptions}\label{sec:The necessity of the analyticity assumptions}
In contrast to the idea of introducing auxiliary conditions \cite{du2022discovery,keller2021discovery}, we rely the analyticity assumptions to ensure the uniqueness of the solution in the function space. We next illustrate this by the following two learning tasks.
\begin{example}\label{example1}
Consider learning the differential equation
\begin{equation*}
\begin{aligned}
&\frac{d}{dt}p=1, \quad \frac{d}{dt}q=\sin{\frac{2\lambda\pi}{h}p},
\end{aligned}
\end{equation*}
with parameter $\lambda \in \mathbb{Z}$ and initial value $(p(0),q(0))=(p_0,q_0)$. The exact solution is given as
\begin{equation*}
\begin{aligned}
&p(t)=p_0+t, \quad q(t)=q_0-\frac{h}{2\lambda\pi}\left(\cos{\frac{2\lambda\pi}{h}(p_0+t)}-\cos{\frac{2\lambda\pi}{h}p_0}\right).
\end{aligned}
\end{equation*}
Taking $t=h$, we have that
\begin{equation*}
\begin{aligned}
&p(h)=p_0+h, \quad
&q(h)=q_0.
\end{aligned}
\end{equation*}
Same states but different parameter $\lambda$ indicate the non-uniqueness of the ODE interpretation of the training data.
\end{example}
\begin{example}\label{example2}
Consider learning the linear equation
\begin{equation*}
\frac{d}{dt} p = -1, \quad p(t)=p(0)-t,
\end{equation*}
by using AB schemes with $M=2$
\begin{equation*}
y_{2} - y_1 = \frac{3h}{2}f(y_1) - \frac{h}{2}f(y_0),
\end{equation*}
we have 
\begin{equation*}
p(2h) - p(h) = \frac{3h}{2}f_{\lambda}(p(h)) - \frac{h}{2}f_{\lambda}(p(0)), \ \text{where}\ f_{\lambda} = -1+\lambda e^{x\ln3/h},
\end{equation*}
which indicates the non-uniqueness of solutions leading to zero loss.
\end{example}
The analyticity assumptions imply the boundedness of the derivatives of each order due to Cauchy's estimate in several variables (see e.g., Section 1.3 of \cite{scheidemann2005introduction}). It is noted that the $k$-th derivatives of the vector fields in \cref{example1} and \cref{example2} are $\mathcal{O}(1/h^k)$. The assumption on the target governing function $f$ excludes the targets containing high-frequency components similar to those in \cref{example1}, implying our analysis is restricted to the discovery of low-frequency dynamics. The assumption on the learned network $f_{\theta}$ excludes the solutions containing high-frequency and is reasonable due to the implicit or explicit regularization.

\subsection{Proofs}\label{sec:proofs}
We present the proofs of \cref{the:error}, \cref{the:error2} and \cref{the:error3} in this subsection. The proofs will be divided into four parts: First, we estimate the upper bounds of the Lie derivatives $\D^k g$ and the expansion terms $f_k$ of the IMDE, which are used as a cornerstone for the subsequent estimations. In addition, we prove that 
\begin{equation*}
\sum_{m=0}^M \alpha_m\phi_{mh,f}(x) - h\sum_{m=0}^M\beta_mf_h^K(\phi_{mh,f}(x))
\end{equation*}
can be made exponentially small by choosing an appropriate truncation index $K$. Following this estimate, we prove \cref{the:error} by estimating the upper bound of the difference between the learned model $f_{\theta}$ and the truncated IMDE $f_h^K$. Finally, we combine the approximation and generalization theories of neural networks
to quantify the learning loss, and thereby conclude \cref{the:error2} and \cref{the:error3}.

We start with the following lemma, which concludes the bound of Lie derivatives.
\begin{lemma}\label{lem: eatimate of Df}
Let $f$ and $g$ be $(Q, r, I)$-analytic on $\Omega$, then
\begin{equation*}
\norm{\D^kg}_{\Omega} \leq k! Q \binom{2k}{k}\brac{\frac{Q}{2r}}^k, \text{ for }k=0,\cdots, I.
\end{equation*}
\end{lemma}
\begin{proof}
Consider
\begin{equation}\label{eq:auxode}
\frac{dz}{dt} = F(z),\quad \text{where } z \in \R,\ F(z) = \frac{Q}{1-z/r}, 
\end{equation}
and define a new Lie derivative
\begin{equation*}
\D_F G(z) =G'(z) F(z). 
\end{equation*}
Here, $F(z)$ satisfies
\begin{equation*}
F^{(k)}(0) = Q \cdot r^{-k} \cdot k!.
\end{equation*}
This fact yields that
\begin{equation}\label{eq: D g and D g}
\norm{\D^kg(x)}_{\Omega} \leq \norm{\D_F^k F(0)}, \text{ for }k=0,\cdots, I.
\end{equation}
In addition, the equation (\ref{eq:auxode}) can be exactly solved:
\begin{equation*}
\phi_{t,F}(0) = r\brac{1- \sqrt{1-\frac{2tQ}{r}}}, \quad F(\phi_{t,F}(0)) = Q\brac{1-\frac{2tQ}{r}}^{-1/2}.
\end{equation*}
According to the expansion that
\begin{equation*}
(1-x)^{-1/2} = \sum_{k=0}^{\infty}\binom{-1/2}{k}(-x)^{k} = \sum_{k=0}^{\infty}\binom{2k}{k}\brac{\frac{x}{4}}^k,
\end{equation*}
we conclude that
\begin{equation*}
F(\phi_{t,F}(0)) =Q\sum_{k=0}^{\infty} \binom{2k}{k}\brac{\frac{tQ}{2r}}^k,
\end{equation*}
which yields that,
\begin{equation*}
\norm{\D_F^k F(0)} = k! Q \binom{2k}{k}\brac{\frac{Q}{2r}}^k.
\end{equation*}
This fact, together with inequality (\ref{eq: D g and D g}) completes the proof.
\end{proof}

Following from Stirling’s formula that
\begin{equation}\label{eq:Stirling’s formula}
\brac{\frac{k}{e}}^k\sqrt{2\pi k} \leq k! \leq \brac{\frac{k}{e}}^k\sqrt{2\pi e k},
\end{equation}
we immediately conclude the following corollary.
\begin{corollary}\label{cor: eatimate of Df}
Let $f$ and $g$ be $(Q, r, I)$-analytic on $\Omega$, then
\begin{equation*}
\norm{\D^kg}_{\Omega} \leq k! Q\brac{\frac{2Q}{r}}^k, \text{ for }k=0,\cdots, I.
\end{equation*}
\end{corollary}

Next, we rewrite the recursion formula (\ref{eq:definition of f_k}) and estimate the coefficients.
\begin{lemma}\label{lem: eatimate of f_k}
The expansion terms $f_k$ defined by (\ref{eq:definition of f_k}) satisfy
\begin{equation*}
f_k(y) = \xi_k (\D^kf)(y),
\end{equation*}
where
\begin{equation}\label{eq: xik}
\begin{aligned}
\xi_0 = 1,\ \xi_k=&\sum_{m=0}^M \alpha_m\frac{m^{k+1}}{(k+1)!}-\sum_{m=0}^M\beta_m\sum_{j=1}^k\frac{m^j}{j!}\xi_{k-j}, \text{ for } k\geq 1.
\end{aligned}
\end{equation}
In addition, for $k\geq 1$, we have $|\xi_k| \leq 2\hat{\mu}/\hat{Z}^k$, where constants $\hat{\mu}$ and $\hat{Z}$ depend only on the coefficients $\alpha_m$ and $\beta_m$ of the LMM.
\end{lemma}
\begin{proof}
We first prove the first part by induction on $k \geq 0$. First, the case when $k=0$ is obvious. Suppose now that the statement holds for $k \leq K-1$. Then, by the recursion formula (\ref{eq:definition of f_k}) and the inductive hypothesis, we have
\begin{equation*}
\begin{aligned}
f_{K}(y)=& \sum_{m=0}^M \alpha_m \frac{m^{K+1}}{(K+1)!}(\D^Kf)(y)-\sum_{m=0}^M\beta_m\sum_{j=1}^K\frac{m^j}{j!}(\D^jf_{K-j})(y)\\
=& \sum_{m=0}^M \alpha_m \frac{m^{K+1}}{(K+1)!}(\D^Kf)(y)-\sum_{m=0}^M\beta_m\sum_{j=1}^K\frac{m^j}{j!}\xi_{K-j}(\D^Kf)(y)\\
=& \xi_{K}(\D^Kf)(y),
\end{aligned}
\end{equation*}
which completes the induction and thus concludes the first part of the proof.

Subsequently we give an estimate of the coefficients $\xi_k$. Let $\Xi(z) = \sum_{k=1}^{\infty} \xi_k z^k$, $z\in \C$. By multiplying (\ref{eq: xik}) with $z^k$ and summing over $k\geq 1$, we obtain that
\begin{equation*}
\begin{aligned}
\Xi(z) = \sum_{k=1}^{\infty}z^k\sum_{m=0}^M \alpha_m \frac{m^{k+1}}{(k+1)!} - \sum_{k=1}^{\infty}z^k\sum_{m=0}^M\beta_m\sum_{j=1}^k\frac{m^j}{j!} \xi_{k-j}.
\end{aligned}
\end{equation*}
By interchanging the summation order, we have
\begin{equation*}
\begin{aligned}
\Xi(z) =& \sum_{m=0}^M \alpha_m \sum_{k=1}^{\infty} \frac{m^{k+1}}{(k+1)!} z^k - \sum_{m=0}^M\beta_m\sum_{k=1}^{\infty}z^k\sum_{j=1}^k\frac{m^j}{j!} \xi_{k-j}\\
=& \sum_{m=0}^M \alpha_m \sum_{k=1}^{\infty} \frac{m^{k+1}}{(k+1)!} z^k - \sum_{m=0}^M\beta_m\sum_{j=1}^{\infty}\frac{m^j}{j!}z^j\sum_{k-j=0}^{\infty} \xi_{k-j} z^{k-j}\\
=& \sum_{m=0}^M \alpha_m (e^{mz}-1-mz)/z + \sum_{m=0}^M\beta_m (1 - e^{mz})(\Xi(z)+1).
\end{aligned}
\end{equation*}
Using the fact that $\sum_{m=0}^M\beta_m = \sum_{m=0}^Mm \cdot \alpha_m=1$, we conclude that
\begin{equation*}
\Xi(z)+1 = \frac{\sum_{m=0}^M \alpha_m (e^{mz}-1)/z}{\sum_{m=0}^M\beta_m e^{mz}}.
\end{equation*}
In addition, since $(\sum_{m=0}^M\beta_m e^{mz})|_{z=0} = 1$, there exists a constant $\hat{Z}<1$ such that 
$\left|\sum_{m=0}^M\beta_m e^{mz}\right| \geq 1/2$ for $|z| \leq \hat{Z}$. Therefore, we have that
\begin{equation*}
\left|\Xi(z)+1\right|\leq 2\hat{\mu}, \ \text{for } |z| \leq \hat{Z},\ \text{where } \hat{\mu} = \max_{|z|\leq \hat{Z}} \left|\sum_{m=0}^M \alpha_m (e^{mz}-1)/z\right|.  
\end{equation*}
By Cauchy's estimate, together with the fact that $k!\xi_k =\Xi^{(k)}(0)$, we conclude that $\xi_k \leq 2\hat{\mu}/\hat{Z}^k$.
\end{proof}

Combining \cref{lem: eatimate of Df} and \cref{lem: eatimate of f_k}, we have the following estimates.
\begin{corollary}\label{cor: eatimate of LD of f_k}
Let $f$ be $(Q, r)$-analytic on $\Omega$. Then, the expansion terms $f_k$ defined by (\ref{eq:definition of f_k}) satisfy
\begin{equation*}
\begin{aligned}
\norm{\D^jf_k}_{\Omega}
\leq & j! \mu Q \brac{\frac{\eta k Q}{r}}^k \brac{\frac{4Q}{r}}^j,
\end{aligned}
\end{equation*}
where constants $\mu$ and $\eta$ depend only on the coefficients $\alpha_m$ and $\beta_m$ of the LMM.
\end{corollary}
\begin{proof}
Combining \cref{lem: eatimate of Df} and \cref{lem: eatimate of f_k}, we obtain that
\begin{equation*}
\begin{aligned}
\norm{\D^jf_k}_{\Omega} = \norm{\xi_k\D^{k+j}f}_{\Omega}\leq & 2\hat{\mu}/\hat{Z}^k (k+j)! Q \binom{2(k+j)}{k+j}\brac{\frac{Q}{2r}}^{k+j}.
\end{aligned}
\end{equation*}
As a consequence of the fact that $(k+j)! \leq k!j!2^{k+j}$ and Stirling’s formula (\ref{eq:Stirling’s formula}), we conclude that
\begin{equation*}
\begin{aligned}
\norm{\D^jf_k}_{\Omega}\leq& 2\hat{\mu}/\hat{Z}^k k!j! Q \binom{2(k+j)}{k+j}\brac{\frac{Q}{r}}^{k+j}
\leq j! \mu Q \brac{\frac{\eta k Q}{r}}^k \brac{\frac{4Q}{r}}^j,
\end{aligned}
\end{equation*}
where $\mu = 2^{3/2}e\hat{\mu}$, $\eta = 4/e\hat{Z}$.
\end{proof}

The following lemma is analogous to Theorem 2 in \cite{hairer1999backward}. The latter shows the rigorous exponentially small error estimates for the modified differential equation of LMMs. Our proof, as well as the proofs of Lemma \ref{lem: eatimate of Df} and \ref{lem: eatimate of f_k}, is also inspired by theirs and other analyses for modified differential equations \cite{benettin1994hamiltonian, hairer1997life, reich1999backward}. 
\begin{lemma}\label{lem: estimate of truncation}
Let $f$ be $(Q, r)$-analytic on $\Omega$. Let $\gamma_0 = r/(e\eta Q)$, and let $K$ be the largest integer satisfying $h K \leq \gamma_0$. If $h\leq \eta \gamma_0/(4 M)$, the truncated IMDE $f_h^K =\sum_{k=0}^K h^kf_k$ satisfies
\begin{equation*}
\norm{\sum_{m=0}^M \alpha_m\phi_{mh,f}(\cdot) - h\sum_{m=0}^M\beta_mf_h^K(\phi_{mh,f}(\cdot))}_{\Omega} \leq \rho hQ e^{-\gamma_0/h},
\end{equation*}
where $\rho$ depends only on the coefficients $\alpha_m$ and $\beta_m$ of the LMM.
\end{lemma}
\begin{proof}
First, by using (\ref{eq:Ld}) with substituting $t=mh$ and $F(y)=f_h(y)$, we have that
\begin{equation*}
\begin{aligned}
&h\sum_{m=0}^M\beta_mf_h^K(\phi_{mh,f}(x))\\
=&h\sum_{m=0}^M\beta_m \sum_{i=0}^{K}\sum_{j=0}^{\infty}\frac{(mh)^j}{j!} h^i(\D^jf_i)(x)\\
=&h\sum_{m=0}^M\beta_m \sum_{i=0}^{K}\sum_{j=0}^{K-i}\frac{(mh)^j}{j!} h^i(\D^jf_i)(x) 
+h\sum_{m=0}^M\beta_m \sum_{i=0}^{K}\sum_{j=K-i+1}^{\infty}\frac{(mh)^j}{j!} h^i(\D^jf_i)(x)\\
=&:\mathcal{F}^{\beta}_{\leq K+1}(x) + \mathcal{F}^{\beta}_{> K+1} (x).
\end{aligned}
\end{equation*}
According to \cref{cor: eatimate of LD of f_k}, we deduce that
\begin{small}
\begin{equation*}
\begin{aligned}
\norm{\mathcal{F}^{\beta}_{> K+1}}_{\Omega} \leq& \mu h Q \sum_{m=0}^M|\beta_m| \sum_{i=0}^{K} \brac{\frac{\eta i h Q}{r}}^i \sum_{j=K-i+1}^{\infty} \brac{\frac{4mhQ}{r}}^j.
\end{aligned}
\end{equation*}
\end{small}
Since $4mhQ/r \leq 1/e$, we have that
\begin{equation*}
\begin{aligned}
\norm{\mathcal{F}^{\beta}_{> K+1}}_{\Omega} \leq& \frac{e\mu}{e-1} h Q \sum_{m=0}^M|\beta_m| \sum_{i=0}^{K} \brac{\frac{\eta i h Q}{r}}^i e^{i-K-1}.
\end{aligned}
\end{equation*}
Subsequently, according to the property of $K$ that $hK\leq \gamma_0$, we obtain that
\begin{equation*}
\begin{aligned}
\norm{\mathcal{F}^{\beta}_{> K+1}}_{\Omega} \leq& \frac{e\mu}{e-1} h Q \sum_{m=0}^M|\beta_m| \sum_{i=0}^{K} \brac{\frac{i}{K}}^i e^{-K-1}.
\end{aligned}
\end{equation*}
By observing the facts
\begin{equation*}
\frac{1}{K}+\brac{\frac{2}{K}}^2>\frac{1}{K+1}+\brac{\frac{2}{K+1}}^2+\brac{\frac{3}{K+1}}^3,\quad \text{for } K\geq 7,
\end{equation*}
and 
\begin{equation*}
\brac{\frac{i}{K}}^i \geq \brac{\frac{i+1}{K+1}}^{i+1},\quad \text{for }1\leq i\leq K,
\end{equation*}
we deduce that the sum $\sum_{i=0}^{K} (i/K)^i$ is maximal for $K=6$ and bounded by $2.01$. Combining these estimates, together with the property of $K$ that $h(K+1) >\gamma_0$, we thus conclude that 
\begin{equation}\label{eq: estimate of F^beta}
\norm{\mathcal{F}^{\beta}_{> K+1}}_{\Omega}\leq \rho_{\beta} h Q e^{-\gamma_0/h},\quad\text{where }\rho_{\beta} = \frac{2.01e\mu}{e-1}\sum_{m=0}^M|\beta_m|. 
\end{equation}

In addition, using formula (\ref{eq:exasolu}), we have
\begin{equation*}
\begin{aligned}
\sum_{m=0}^M \alpha_m\phi_{mh,f}(x) =& \sum_{m=0}^M \alpha_m \sum_{k=0}^{K+1}\frac{(mh)^k}{k!}(\D^kI_D)(x)+ \sum_{m=0}^M \alpha_m \sum_{k=K+2}^{\infty}\frac{(mh)^k}{k!}(\D^kI_D)(x)\\
=&: \mathcal{F}^{\alpha}_{\leq K+1} + \mathcal{F}^{\alpha}_{> K+1}.
\end{aligned}
\end{equation*}
By \cref{cor: eatimate of Df} and the fact that $\D^k I_D = \D^{k-1}f$, we deduce that
\begin{equation}\label{eq: estimate of F^alpha}
\begin{aligned}
\norm{\mathcal{F}^{\alpha}_{> K+1}}_{\Omega} \leq & hQ\sum_{m=0}^M |\alpha_m|m\sum_{k=K+2}^{\infty} \brac{\frac{2mhQ}{r}}^{k-1}\\
\leq& \rho_{\alpha}hQ e^{-\gamma_0/h},
\end{aligned}
\end{equation}
with $\rho_{\alpha} = \frac{2e }{2e -1} \sum_{m=0}^M 2^{- 4M/\eta}|\alpha_m|m $. Here, we again use the fact that 
\begin{equation*}
4mhQ/r \leq 1/e, \ K+1\geq 4M/\eta,\ \text{and}\ e^{-(K+1)} \leq e^{-\gamma_0/h}.  
\end{equation*}

Finally, due to the definition of IMDE, we have that
\begin{equation*}
\sum_{m=0}^M \alpha_m\phi_{mh,f}(x) - h\sum_{m=0}^M\beta_mf_h^K(\phi_{mh,f}(x)) = \mathcal{F}^{\alpha}_{> K+1} + \mathcal{F}^{\beta}_{> K+1}.
\end{equation*}
Therefore, by combining estimates (\ref{eq: estimate of F^beta}) and (\ref{eq: estimate of F^alpha}) we conclude that
\begin{equation*}
\begin{aligned}
\norm{\sum_{m=0}^M \alpha_m\phi_{mh,f}(\cdot) - h\sum_{m=0}^M\beta_mf_h^K(\phi_{mh,f}(\cdot))}_{\Omega} \leq& \norm{\mathcal{F}^{\alpha}_{> K+1}}_{\Omega} + \norm{\mathcal{F}^{\beta}_{> K+1}}_{\Omega}\\
\leq& \rho hQ e^{-\gamma_0/h},
\end{aligned}
\end{equation*}
where $\rho = \rho_{\alpha} + \rho_{\beta}$ depends only on the coefficients $\alpha_m$ and $\beta_m$.
\end{proof}

Again using Lemma \ref{lem: eatimate of f_k}, we can give the following lemma, which yields upper bounds of $f_h^K$ and $f_h^K -f$ immediately.
\begin{lemma}\label{lem: estimate of f_h}
Let $f$ be $(Q, r)$-analytic on $\Omega$. Let $\gamma_0 = r/(e\eta Q)$, and let $K$ be the largest integer satisfying $h K \leq \gamma_0$. Then for integer $p$ with $0\leq p\ \leq K$, we have
\begin{equation*}
\norm{\D^j\sum_{k=p}^K h^k f_k}_{\Omega} \leq j! \mu Q \brac{\frac{4Q}{r}}^j \brac{\frac{\eta Q}{r}}^p d_p h^p,
\end{equation*}
where $d_p$ depends only on $p$.
\end{lemma}
\begin{proof}
From \cref{cor: eatimate of LD of f_k}, we obtain that
\begin{equation*}
\begin{aligned}
\norm{\D^j\sum_{k=p}^K h^k f_k}_{\Omega} \leq &  j! \mu Q \brac{\frac{4Q}{r}}^j\sum_{k=p}^K h^k \brac{\frac{\eta k Q}{r}}^k.
\end{aligned}
\end{equation*}
Since $h \leq r/(Ke\eta Q)$, we deduce that
\begin{equation*}
\begin{aligned}
\norm{\D^j\sum_{k=p}^K h^k f_k}_{\Omega} \leq & h^p j! \mu Q \brac{\frac{4Q}{r}}^j \brac{\frac{\eta Q}{r}}^p\sum_{k=p}^K \brac{\frac{k}{eK}}^{k} e^pK^p.
\end{aligned}
\end{equation*}
By observing that $\brac{\frac{k}{eK}}^{k}$ is decreasing for $k$ on the interval $[p+1,K]$, we have that
\begin{equation*}
\begin{aligned}
\sum_{k=p}^K \brac{\frac{k}{eK}}^{k} \leq& \brac{\frac{p}{eK}}^{p} + (K-p)\brac{\frac{p+1}{eK}}^{p+1}\\
\leq& \left[ \brac{\frac{p}{e}}^{p} + \brac{\frac{p+1}{e}}^{p+1}\right] K^{-p} =: d_p e^{-p} K^{-p},
\end{aligned}
\end{equation*}
and thereby conclude the proof.
\end{proof}

\cref{lem: estimate of f_h} yields the following corollarys immediately.
\begin{corollary}\label{cor: eatimate of f_h^K}
Under notations and conditions of Lemma \ref{lem: estimate of f_h}, the truncated IMDE $f_h^K =\sum_{k=0}^K h^kf_k$ satisfies
\begin{equation*}
\norm{\D^j f_h^K}_{\Omega} \leq (1+e^{-1}) j! \mu Q \brac{\frac{4Q}{r}}^j
\end{equation*}
\end{corollary}

\begin{corollary}\label{cor: eatimate of f_h^K-f}
Under notations and conditions of Lemma \ref{lem: estimate of f_h}, if the LMM is of order $p$, then, the truncated IMDE $f_h^K =\sum_{k=0}^K h^kf_k$ satisfies
\begin{equation*}
\norm{f_h^K-f}_{\Omega} \leq \mu Q \brac{\frac{\eta Q}{r}}^p d_p h^p,
\end{equation*}
where $d_p$ depends only on $p$.
\end{corollary}

{
With these results, we are able to provide the proof of \cref{the:error}.
\begin{proof}[Proof of \cref{the:error}]
Let us define for $x \in \mathcal{B}(\K, R)$
\begin{equation*}
\begin{aligned}
\ell(x): =& \sum_{m=0}^M h^{-1}\alpha_m\phi_{mh,f}(x) - \sum_{m=0}^M\beta_mf_{\theta^*}(\phi_{mh,f}(x))\\
\delta(x) : =& \sum_{m=0}^M\beta_mf_h^K(\phi_{mh,f}(x)) - \sum_{m=0}^M\beta_mf_{\theta^*}(\phi_{mh,f}(x)).
\end{aligned}
\end{equation*}
From \cref{lem: estimate of truncation}, we know that for sufficiently small $h$,
\begin{equation}\label{eq:estimate of delta}
\norm{\delta(x)} \leq \norm{\ell(x)} + \rho Q e^{-\gamma_0/h}, \quad \forall x \in \mathcal{B}(\K, R).
\end{equation}
Take $c>1$ obeying 
\begin{equation*}
\sum_{m=0}^M |\beta_m| \frac{\sqrt{\pi}e^{5/4}}{c-1} \leq \frac{1}{e}\brac{1 - \frac{1}{c^2}\sum_{m=0}^M|\beta_m|},
\end{equation*}
and denote
\begin{equation*}
F(\tau) := f_h^K(\phi_{\tau,f}(x)) - f_{\theta^*}(\phi_{\tau,f}(x)).
\end{equation*}
Let $\tilde{I}$ be the largest integer satisfying $4Q c\tilde{I}^{5/2}Mh\leq r$ and let $h_1 = c \sqrt{\tilde{I}} Mh$. By interpolating the function $F(\tau)$ at points $-\tilde{I}h_1, (1-\tilde{I})h_1, \cdots, \tilde{I}h_1$, we obtain that
\begin{equation*}
F(mh) = \sum_{i=-\tilde{I}}^{\tilde{I}} F(i h_1)\prod_{\substack{j=-\tilde{I}\\j\neq i}}^{\tilde{I}} \frac{mh - jh_1}{ih_1 - jh_1} + \frac{ \D^{2\tilde{I}+1}F(\xi)}{(2\tilde{I}+1)!}\prod_{i=-\tilde{I}}^{\tilde{I}} (mh - jh_1): = \sum_{i=-\tilde{I}}^{\tilde{I}} F_i^m + F_{res}^m.
\end{equation*}
By observing the fact that
\begin{equation*}
\prod_{i=-\tilde{I}}^{\tilde{I}} |mh - jh_1| = mhh_1^{2\tilde{I}}\prod_{i=1}^{\tilde{I}} (j-\frac{m}{c\sqrt{\tilde{I}}M})\prod_{i=1}^{\tilde{I}} (j+\frac{m}{c\sqrt{\tilde{I}}M}) \leq mhh_1^{2\tilde{I}} (\tilde{I}!)^2 
\end{equation*}
and combining \cref{cor: eatimate of Df} and \cref{cor: eatimate of f_h^K}, we deduce that if the learned vector field $f_{\theta^*}$ is $(Q, r, 2\tilde{I}+1)$-analytic on $\mathcal{B}(\K, R)$, then, for any $x \in \mathcal{B}(\K, R)$,
\begin{equation*}
\begin{aligned}
\norm{F_{res}^m} \leq \brac{(1+e^{-1})\mu +1} Q \brac{\frac{4Q}{r}}^{2\tilde{I}+1}mhh_1^{2\tilde{I}} (\tilde{I}!)^2 
\leq C_{res} Q \brac{\frac{4Q c \tilde{I}^{3/2}Mh}{re}}^{2\tilde{I}+1},
\end{aligned}
\end{equation*}
where $C_{res} = \brac{(1+e^{-1})\mu +1} 2\pi e^2/c $. It then follows from the definition of $\tilde{I}$, that
\begin{equation}\label{eq: estimation of F_res}
\norm{F_{res}^m} \leq C_{res}Q e^{-\gamma_1/h^{2/5}}, \text{ for }\forall x \in \mathcal{B}(\K, R), \text{ where } \gamma_1 = 2 \brac{\frac{r}{4Q cM}}^{2/5}.
\end{equation}
Next, for $i\neq 0$, we have
\begin{equation*}
\begin{aligned}
\norm{F_i^m} =& \norm{F(ih_1)} \frac{mh}{ih_1-mh} \prod_{j=1}^{\tilde{I}} (j+\frac{m}{c\sqrt{\tilde{I}}M})\prod_{j=1}^{\tilde{I}}(j-\frac{m}{c\sqrt{\tilde{I}}M}) \prod_{j=-\tilde{I}}^{i-1} \frac{1}{i-j} \prod_{j=i+1}^{\tilde{I}} \frac{1}{i-j} \\
\leq & \norm{F(ih_1)} \frac{1}{c\sqrt{\tilde{I}}-1}(\tilde{I}!)^2\frac{1}{(\tilde{I}+i)!(\tilde{I}-i)!},\quad \forall x \in \mathcal{B}(\K, R).
\end{aligned}
\end{equation*}
Denote 
\begin{equation*}
\K_{\tau} = \{ \phi_{t, f}(x)| x \in \K, -\tau \leq t \leq \tau\} \subseteq \mathcal{B}(\K, Q\tau).
\end{equation*}
By integrating over $\K_{n\tilde{I}h_1}$ with $n \leq R/(Q\tilde{I}h_1)-1$, we obtain that
\begin{equation*}
\begin{aligned}
&\int_{\K_{n\tilde{I}h_1}} \sum_{\substack{i=-\tilde{I}\\i\neq 0}}^{\tilde{I}} \norm{F_i^m} dx \\
\leq &\sum_{\substack{i=-\tilde{I}\\i\neq 0}}^{\tilde{I}} \frac{1}{c\sqrt{\tilde{I}}-1}\frac{(\tilde{I}!)^2}{(\tilde{I}+i)!(\tilde{I}-i)!} \int_{\K_{n\tilde{I}h_1}} \norm{f_h^K(\phi_{ih_1,f}(x)) - f_{\theta^*}(\phi_{ih_1,f}(x))} dx \\
\leq& \sum_{i=0}^{2\tilde{I}} \frac{1}{c\sqrt{\tilde{I}}-1}\frac{(\tilde{I}!)^2}{i!(2\tilde{I}-i)!} \int_{\K_{(n+1)\tilde{I}h_1}} \norm{f_h^K(x) - f_{\theta^*}(x)} \det \frac{\partial \phi_{(\tilde{I}-i)h_1, f}(x)}{\partial x}dx.
\end{aligned}
\end{equation*}
As a consequence of the fact that $i!(2\tilde{I}-i)! 2^{2\tilde{I}} \geq (2\tilde{I})!$ and the estimation that
\begin{equation*}
0\leq \det \frac{\partial \phi_{(\tilde{I}-i)h_1, f}(x)}{\partial x} = e^{\int_{0}^{(\tilde{I}-i)h_1} \text{trace}{f'(x(\tau))}d\tau} \leq e^{\tilde{I}h_1Q/r} \leq e^{1/4},
\end{equation*}
we deduce that
\begin{equation}\label{eq: estimation of F_i}
\begin{aligned} 
\int_{\K_{n\tilde{I}h_1}} \sum_{\substack{i=-\tilde{I}\\i\neq 0}}^{\tilde{I}} \norm{F_i^m} dx\leq & \frac{e^{1/4}}{c\sqrt{\tilde{I}}-1}\frac{2^{2\tilde{I}}(\tilde{I}!)^2}{(2\tilde{I})!} \int_{\K_{(n+1)\tilde{I}h_1}} \norm{f_h^K(x) - f_{\theta^*}(x)} dx \\
\leq & \frac{\sqrt{\pi}e^{5/4}}{c-1}\int_{\K_{(n+1)\tilde{I}h_1}} \norm{f_h^K(x) - f_{\theta^*}(x)} dx,
\end{aligned}
\end{equation}
where the second inequality follows from Stirling’s formula (\ref{eq:Stirling’s formula}).
Subsequently, for $i=0$, we have
\begin{equation}\label{eq: estimation of F_0}
\begin{aligned}
\norm{\sum_{m=0}^M\beta_mF_0^m} =& \norm{\sum_{m=0}^M\beta_mF(0)\prod_{j=-\tilde{I}}^{-1} \frac{\frac{m}{c\sqrt{\tilde{I}}M} - j}{ - j}\prod_{j=1}^{\tilde{I}} \frac{\frac{m}{c\sqrt{\tilde{I}}M} - j}{ - j}}\\
= & \norm{F(0)}\left|\sum_{m=0}^M\beta_m\prod_{j=1}^{\tilde{I}} \frac{\frac{m}{c\sqrt{\tilde{I}}M} + j}{ j}\prod_{j=1}^{\tilde{I}} \frac{j - \frac{m}{c\sqrt{\tilde{I}}M}}{j}\right|\\
\geq & \brac{1 - \frac{1}{c^2}\sum_{m=0}^M|\beta_m|}\norm{F(0)},
\end{aligned}
\end{equation}
where the last inequality follows from the estimation that
\begin{equation*}
\begin{aligned}
&\left|\sum_{m=0}^M\beta_m\prod_{j=1}^{\tilde{I}} \frac{j+\frac{m}{c\sqrt{\tilde{I}}M} }{ j}\prod_{j=1}^{\tilde{I}} \frac{j - \frac{m}{c\sqrt{\tilde{I}}M}}{j}\right| \\
\geq & 1 - \left|\sum_{m=0}^M\beta_m\brac{1-\prod_{j=1}^{\tilde{I}} \brac{1 - \frac{m^2}{c^2 \tilde{I} M^2 j^2}}}\right|\\
\geq& 1 - \sum_{m=0}^M|\beta_m|+\sum_{m=0}^M|\beta_m|\prod_{j=1}^{\tilde{I}} \brac{1 - \frac{m^2}{c^2 \tilde{I} M^2 j^2}}\\
\geq& 1 - \sum_{m=0}^M|\beta_m|+\sum_{m=0}^M|\beta_m|\brac{1 - \frac{1}{c^2 \tilde{I} }}^{\tilde{I}}
\geq 1 - \frac{1}{c^2}\sum_{m=0}^M|\beta_m|.
\end{aligned}
\end{equation*}
According to the definition of $\delta$ and $F(\tau)$, we have 
\begin{equation}\label{eq:estimation before integrating}
\norm{\sum_{m=0}^M \beta_m F_0^m}\leq \norm{\delta(x)} + \sum_{m=0}^M\sum_{\substack{i=-\tilde{I}\\i\neq 0}}^{\tilde{I}} |\beta_m| \norm{F_i^m} + \sum_{m=0}^M |\beta_m| \norm{F_{res}^m}.
\end{equation}
By integrating (\ref{eq:estimation before integrating}) over $\K_{n\tilde{I}h_1}$ with $n \leq R/(Q\tilde{I}h_1)-1$, we deduce that
\begin{equation*}
\begin{aligned}
 &\int_{\K_{n\tilde{I}h_1}}\norm{\sum_{m=0}^M \beta_m F_0^m}dx \\
\leq& \int_{\K_{n\tilde{I}h_1}} \norm{\delta(x)} dx+ \sum_{m=0}^M |\beta_m|\int_{\K_{n\tilde{I}h_1}} \sum_{\substack{i=-\tilde{I}\\i\neq 0}}^{\tilde{I}} \norm{F_i^m} dx +  \sum_{m=0}^M |\beta_m| \int_{\K_{n\tilde{I}h_1}}\norm{F_{res}^m} dx.
\end{aligned}
\end{equation*}
Subsequently, combining the estimates (\ref{eq:estimate of delta}), (\ref{eq: estimation of F_res}), (\ref{eq: estimation of F_i}) and (\ref{eq: estimation of F_0}), we conclude that
\begin{equation*}
\begin{aligned}
&\brac{1 - \frac{1}{c^2}\sum_{m=0}^M|\beta_m|}\int_{\K_{n\tilde{I}h_1}}\norm{f_h^K(x) - f_{\theta^*}(x)}dx\\
\leq& \int_{\K_{n\tilde{I}h_1}} \norm{\ell(x)} dx
+ \sum_{m=0}^M |\beta_m| \frac{\sqrt{\pi}e^{5/4}}{c-1}\int_{\K_{(n+1)\tilde{I}h_1}} \norm{f_h^K(x) - f_{\theta^*}(x)} dx \\
&+ \mathbf{V}\sum_{m=0}^M |\beta_m| C_{res}Q e^{-\gamma_1/h^{2/5}} + \mathbf{V}\rho Q e^{-\gamma_0/h},
\end{aligned}
\end{equation*}
where $\mathbf{V} = \int_{\mathcal{B}(\K, R)}1 dx$ is the volume of the domain $\mathcal{B}(\K, R)$. This estimate yields that for $n \leq R/(Q\tilde{I}h_1)-1$,
\begin{equation*}
\begin{aligned}
&\int_{\K_{n\tilde{I}h_1}}\norm{f_h^K(x) - f_{\theta^*}(x)} dx - \lambda_1\int_{\mathcal{B}(\K, R)} \norm{\ell(x)} dx -\lambda_2e^{-\gamma_1/h^{2/5}}\\
\leq& e^{-1} \brac{\int_{\K_{(n+1)\tilde{I}h_1}}\norm{f_h^K(x) - f_{\theta^*}(x)} dx - \lambda_1 \int_{\mathcal{B}(\K, R)} \norm{\ell(x)} dx - \lambda_2e^{-\gamma_1/h^{2/5}}},
\end{aligned}
\end{equation*}
where
\begin{equation*}
\lambda_1 = \frac{e}{e-1}\frac{1}{1 - \frac{1}{c^2}\sum_{m=0}^M|\beta_m|},\ 
\lambda_2 = \lambda_1 \mathbf{V} Q\brac{\sum_{m=0}^M |\beta_m| C_{res} + \rho e^{\frac{\gamma_1h^{3/5}-\gamma_0}{h}}}.
\end{equation*}
By using this estimate iteratively, together with \cref{cor: eatimate of f_h^K}, we deduce that
\begin{equation*}
\begin{aligned}
\int_{\K}\norm{f_h^K(x) - f_{\theta^*}(x)} dx \leq& \lambda_1\int_{\mathcal{B}(\K, R)} \norm{\ell(x)} dx +\lambda_2e^{-\gamma_1/h^{2/5}} + \lambda_3 e^{-R/(Q\tilde{I}h_1)}, 
\end{aligned}
\end{equation*}
where $\lambda_3 = \brac{(1+e^{-1})\mu +1} \mathbf{V} Q$. Moreover, the definition of $\tilde{I}$ yields that
\begin{equation*}
R/(Q \tilde{I} h_1) \geq 4^{3/5}(R/Q)(r/Q)^{-3/5}(cM)^{-2/5}h^{-2/5}, 
\end{equation*}
and we thus conclude that there exist constants $\gamma$ and $c_0$ depended only on $r/Q$, $R/Q$, dimension $D$, domain $\K$ and the coefficients $\alpha_m$ and $\beta_m$, such that
\begin{equation}\label{eq :final estimate}
\begin{aligned}
\int_{\K}\norm{ f_{\theta^*}(x)- f_h^K(x) } dx \leq c_0 \brac{\int_{\mathcal{B}(\K, R)} \norm{\ell(x)}_2^2 dx}^{1/2} +c_1 e^{-\gamma/h^{2/5}}.
\end{aligned}
\end{equation}

The bound of $\int_{\K} \norm{f_{\theta^*} - f} dx$ is an immediate consequence following from estimates (\ref{eq :final estimate}), \cref{cor: eatimate of f_h^K-f} and the triangle inequality. The proof is completed.
\end{proof}

Finally, we give the proofs of \cref{the:error2} and \cref{the:error3} by combining the approximation and generalization theories of neural networks.
\begin{proof}[Proof of \cref{the:error2}]
It suffices to estimate $\mathcal{L}$ according to \cref{the:error}.
As a consequence of \cref{thm:gen}, with taking $\mathcal{G}(f_{\theta})(\cdot) = \sum_{m=0}^M\beta_mf_{\theta}(\phi_{mh,f}(\cdot))$ and $u = \sum_{m=0}^M h^{-1}\alpha_m\phi_{mh,f}$, we have
\begin{equation}\label{eq:estimate of L}
 \mathcal{L}\leq C_N \mathcal{L}_{h,N}(f_{\theta^*}) + \brac{\hat{\lambda}_{N}\lip{\mathcal{G}(f_{\theta^*})}_{\mathcal{B}(\K, R)}^2+C'}N^{-\frac{1}{D}}.
\end{equation}
Subsequently, by \cref{thm:approximation}, we know that there is a tanh FNN $f_{\widehat{\theta}} \in \M$ such that
\begin{equation}\label{eq: estimate of f_h^K - f^N,s}
\norm{f_h^K -f_{\widehat{\theta}}}_{\mathcal{B}(\K, R)} \leq C''(1+\delta) Q \left(\frac{3D}{2r\widehat{N}}\right)^{s},
\end{equation}
where constant $C''$ depends only on the cuboid $\mathcal{C}$. Note that the widths are increased by a factor of $D$ times since $f_h^K(x) \in \R^D$. Since $\mathcal{L}_{h,N} (f_{\theta})$ be the global minimizer of $\mathcal{L}_N$ (\ref{eq:loss-of-lmnets}), we have
\begin{equation}\label{eq:estimate of L_{h,N}}
\begin{aligned}
\mathcal{L}_{h,N}(f_{\theta^*}) \leq& \mathcal{L}_{h,N} (f_{\widehat{\theta}})\\
\leq& 2 \mathcal{L}_{h,N}(f_h^K) + \frac{2}{N}\sum_{n=1}^N\brac{\sum_{m=0}^M|\beta_m| \norm{f_h^K (x_n) - f_{\widehat{\theta}}(x_n)}_2}^2\\
\leq & 2\rho^2 Q^2 e^{-2\gamma_0/h} + 2 \brac{\sum_{m=0}^M|\beta_m|C''(1+\delta) Q \left(\frac{3D}{2r\widehat{N}}\right)^{s}}^2,
\end{aligned}
\end{equation}
where the last inequality follows from \cref{lem: estimate of truncation} and estimate (\ref{eq: estimate of f_h^K - f^N,s}).
In addition, we have 
\begin{equation}\label{eq:estimate of lip}
\lip{\mathcal{G}(f_{\theta^*})}_{\mathcal{B}(\K, R)} \leq \sum_{m=0}^M|\beta_m| \norm {\frac{\partial f_{\theta}(\phi_{mh,f}(x))}{\partial x}}_{\mathcal{B}(\K, R)}\leq\sum_{m=0}^M |\beta_m| \frac{Q}{r} e^{\frac{MhQ}{r}}. 
\end{equation}
By substituting (\ref{eq:estimate of L_{h,N}}) and (\ref{eq:estimate of lip}) into (\ref{eq:estimate of L}), together with the facts that $C_N=\mathcal{O}(N^{1/2})$ and $N^{1/(2D)+1/4} \leq e^{ \gamma_0/h} / (\rho Q)$, we conclude the proof.
\end{proof}

\begin{proof}[Proof of \cref{the:error3}]
By \cref{thm:approximation}, we have that there is a tanh FNN $f_{\widehat{\theta}} \in \M$ whose weights scale as $\mathcal{O}\brac{\mathbf{C}^{-s/2}\widehat{N}^{D(D+s^2+I^2)/2}\cdot (s(s+2))^{3s(s+2)}}$, such that 
\begin{equation*}
\norm{f_h^K -f_{\widehat{\theta}}}_{\mathcal{B}(\K, R)} \leq C''_0(1+\delta) Q \left(\frac{3D}{2r\widehat{N}}\right)^{s},
\end{equation*}
and for $i=1, \cdots, I$, 
\begin{equation*}
\norm{f^{(i)}-f_{\widehat{\theta}}^{(i)}}_{\mathcal{B}(\K, R)} \leq C''_I \alpha \max\left\{r^{I},\ln^{I}\left(\beta \widehat{N}^{s+D+2}\right)\right\} \cdot \frac{\mathbf{C}(D,I,s,f)}{\widehat{N}^{s-I}}.
\end{equation*}
In addition, the weights $\widehat{\theta} = \{\widehat{W}_l,\widehat{b}_l\}_{l=1}^3$ satisfies
\begin{equation*}
\max \widehat{\theta} \leq C''' \mathbf{C}^{-s/2}\widehat{N}^{D(D+s^2+k^2)/2}(s(s+2))^{3s(s+2)},
\end{equation*}
where $C''_0$ and $C'''$ are constants dependent only on the cuboid $\mathcal{C}$, $C''_I$ is a constant dependent on $\mathcal{C}$ and $I$.
Due to these estimates, we have that
\begin{equation*}
\begin{aligned}
\mathcal{R}_{\varepsilon, r_1, I}(\widehat{\theta})
\leq & \sum_{i=0}^{I} \frac{r_1^i}{i!} \brac{\frac{Qi!}{r^i} + C''_I \alpha \max\left\{r^{I},\ln^{I}\left(\beta \widehat{N}^{s+D+2}\right)\right\} \cdot \frac{\mathbf{C}(D,I,s,f)}{\widehat{N}^{s-I}}} + C'''\\
\leq & 2Q + e^{r_1} C''_I \alpha \max\left\{r^{I},\ln^{I}\left(\beta \widehat{N}^{s+D+2}\right)\right\} \cdot \frac{\mathbf{C}(D,I,s,f)}{\widehat{N}^{s-I}} + C'''.
\end{aligned}
\end{equation*}
In addition, it has been shown in (\ref{eq:estimate of L_{h,N}}) that
\begin{equation*}
\mathcal{L}_{h,N} (f_{\widehat{\theta}})
\leq 2\rho^2 Q^2 e^{-2\gamma_0/h} + 2 \brac{\sum_{m=0}^M|\beta_m|C''(1+\delta) Q \left(\frac{3D}{2r\widehat{N}}\right)^{s}}^2.
\end{equation*}
Since $f_{\theta^*}^{\mathcal{R}}$ be the global minimizer of $\mathcal{L}_N$ (\ref{eq:regularized training loss}), we have
\begin{equation}\label{eq:estimate of regularized loss}
\begin{aligned}
&\mathcal{L}_{h,N} (f_{\theta^*}^{\mathcal{R}}) + N^{-1/D-1/2}\mathcal{R}_{\varepsilon, r_1, I}(\theta^*) \leq \mathcal{L}_{h,N} (f_{\widehat{\theta}}) + N^{-1/D-1/2}\mathcal{R}_{\varepsilon, r_1, I}(\widehat{\theta}).
\end{aligned}
\end{equation}
In particular,
\begin{equation*}
\begin{aligned}
\mathcal{R}_{\varepsilon, r_1, I}(\theta^*) \leq N^{1/D+1/2} \mathcal{L}_{h,N} (f_{\widehat{\theta}}) + \mathcal{R}_{\varepsilon, r_1, I}(\widehat{\theta}).
\end{aligned}
\end{equation*}
By choosing sufficiently large $\widehat{N}$, i.e., a suitable definition of $\widehat{N}_0(N)$, we can obtain that
\begin{equation*}
\mathcal{R}_{\varepsilon, r_1, I}(\theta^*) \leq Q^*, \quad \mathcal{R}_{\varepsilon, r_1, I}(\widehat{\theta}) \leq \widehat{Q}.
\end{equation*}
where $Q^*$ and $\widehat{Q}$ are constants dependent on $Q$ and the cuboid $\mathcal{C}$.
Observing the fact that $\lip{\norm{f_{\theta}^{(i)}(x)}}_{\mathcal{B}(\K, R)}$ depends only the weight $\theta^* = \{W_l^*,b_l^*\}_{l=1}^3$ and 
\begin{equation*}
\max \theta^* \leq Q^* \mathbf{C}^{-s/2}\widehat{N}^{D(D+s^2+k^2)/2}.(s(s+2))^{3s(s+2)},
\end{equation*}
we can choose sufficiently small $\varepsilon$ such that
\begin{equation*}
\sqrt{D}/2 \cdot \lip{\norm{f_{\theta}^{(i)}(x)}}_{\mathcal{B}(\K, R)} \cdot \varepsilon \leq \frac{i!}{r_1^i}, \ \text{for } i=0, \cdots, I.
\end{equation*}
Therefore we obtain that
\begin{equation*}
\begin{aligned}
\norm{f_{\theta}^{(i)}(x)}_{\mathcal{B}(\K, R)} \leq& \max\left\{\norm{f_{\theta}^{(i)}(z_j)} \right\}_{j=1}^J + \sqrt{D}/2 \cdot \lip{\norm{f_{\theta}^{(i)}(x)}}_{\mathcal{B}(\K, R)} \cdot \varepsilon\\
\leq& \frac{i!(Q^*+1)}{r_1^i}, \ \text{for } i=0, \cdots, I. 
\end{aligned}
\end{equation*}
which yields $f_{\theta}$ is $(Q^*+1, r_1)$-analytic on $\mathcal{B}(\K, R)$. 
Subsequently, we can apply \cref{the:error} and it suffices to estimate $\mathcal{L}_{h,N}(f_{\theta^*})$ according to \cref{thm:gen}. 
Again using (\ref{eq:estimate of regularized loss}), we have
\begin{equation*}
\begin{aligned}
\mathcal{L}_{h,N}(f_{\theta^*}) \leq & \mathcal{L}_{h,N} (f_{\widehat{\theta}}) + N^{-1/D-1/2}\mathcal{R}_{\varepsilon, r_1, I}(\widehat{\theta})\\
\leq & 2\rho^2 Q^2 e^{-2\gamma_0/h} + 2 \brac{\sum_{m=0}^M|\beta_m|C''(1+\delta) Q \left(\frac{3D}{2r\widehat{N}}\right)^{s}}^2 + \widehat{Q}N^{-1/D-1/2}.
\end{aligned}
\end{equation*}
The proof is completed.
\end{proof}
}

\section{Numerical experiments}\label{sec: numerical results}
In this section, we present various numerical evidence to demonstrate our theoretical findings on the discovery of hidden dynamics using LMNets. Following \cite{du2022discovery, keller2021discovery,raissi2018multistep}, we consider three benchmark problems, including the damped oscillator problem, the Lorenz system, and the glycolytic oscillator; however, we use different settings from prior results \cite{du2022discovery, keller2021discovery,raissi2018multistep} to avoid repetition and highlight our analysis. Here we test the error orders of AB, BDF, and AM schemes for discovery, whose coefficients can be found in \cite{hairer1993solving}.

As \cref{the:error2} and \cref{the:error3} are direct consequences of \cref{the:error} and existing approximation and generalization theories, we only illustrate \cref{the:error2} and \cref{the:error3} in the first example, and mainly focus on \cref{the:error} in the remaining examples. Herein, we measure the learning loss a posteriori using the test loss, where the test data $\mathcal{T}_{test}$ are sampled in the same form and manner as the training data, i.e.,
\begin{equation*}
\mathcal{T}_{test} = \left\{\big(x_n^{test}, \phi_{h,f}(x_n^{test}), \cdots, \phi_{Mh,f}(x_n^{test})\big)\right\}_{n=1}^{N_{test}},
\end{equation*}
and the test loss is formulated as 
\begin{equation*}
\text{Test loss} = \frac{1}{N_{test}}\sum_{n=1}^{N_{test}}\normt{\sum_{m=0}^M h^{-1} \alpha_m\phi_{mh,f}(x_n^{test}) - \sum_{m=0}^M\beta_mf_{\theta}(\phi_{mh,f}(x_n^{test}))}^2.
\end{equation*}
In practice, we can divide our data into two sets: a training dataset for loss optimization and a test dataset for learning performance measurement. Moreover, we evaluate the average error in the $\ell_{1}$-norm, i.e., 
\begin{equation}\label{eq:error of experiments}
\text{Error}(g_1, g_2) = \frac{1}{|\mathcal{T}|}\sum_{x_i\in \mathcal{T}} \norm{g_1(x_i) - g_2(x_i)}, 
\end{equation}
where $\mathcal{T}$ is composed of several points randomly sampled from the given compact region in the first example, and is composed of all the initial states in the test data, i.e., $\mathcal{T} = \{x_n^{test}\}_{n=1}^{N_{test}}$, in the second and third examples. In all examples, we use the fourth-order RK method on a very fine mesh to generate data and simulate the discovered systems for numerical validation. All neural network models are trained on the PyTorch framework with Tesla P40 GPU, where the Python 3.8 environment and CUDA 10.1 toolkit are utilized. And the code accompanying this paper is publicly available at \url{https://github.com/Aiqing-Zhu/IMDE-LMM}.

\subsection{Damped oscillator problem}\label{sec:do}
To begin with, we consider the damped harmonic oscillator with cubic dynamics. The system of equation for $y=(p,q)$ is given by
\begin{equation}\label{eq:do}
\frac{d}{dt}p = -0.1p^3+2.0q^3,\quad \frac{d}{dt}q = -2.0p^3-0.1q^3.
\end{equation}

\begin{figure}[ht]
  \centering
  \includegraphics[width=1\textwidth]{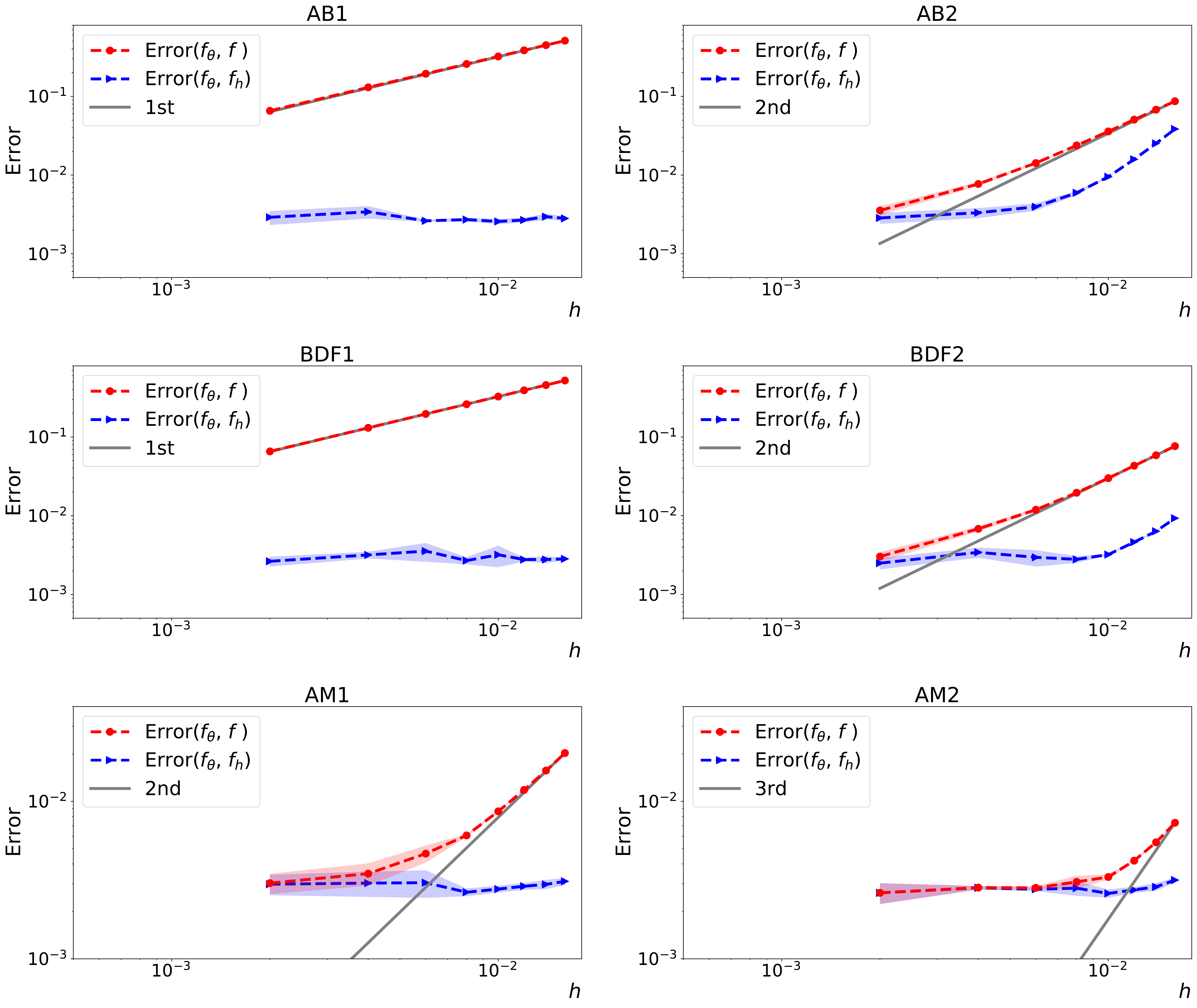}
  \caption{$\text{Error}(f_{\theta}, f)$ and $\text{Error}(f_{\theta}, f_h)$ versus $h$ for the damped harmonic oscillator (\ref{eq:do}). The results are obtained by taking the mean of $5$ independent experiments, and the shaded region represents one standard deviation.}
  \label{fig:DO}
\end{figure}

\begin{figure}[ht]
  \centering
  \includegraphics[width=1\textwidth]{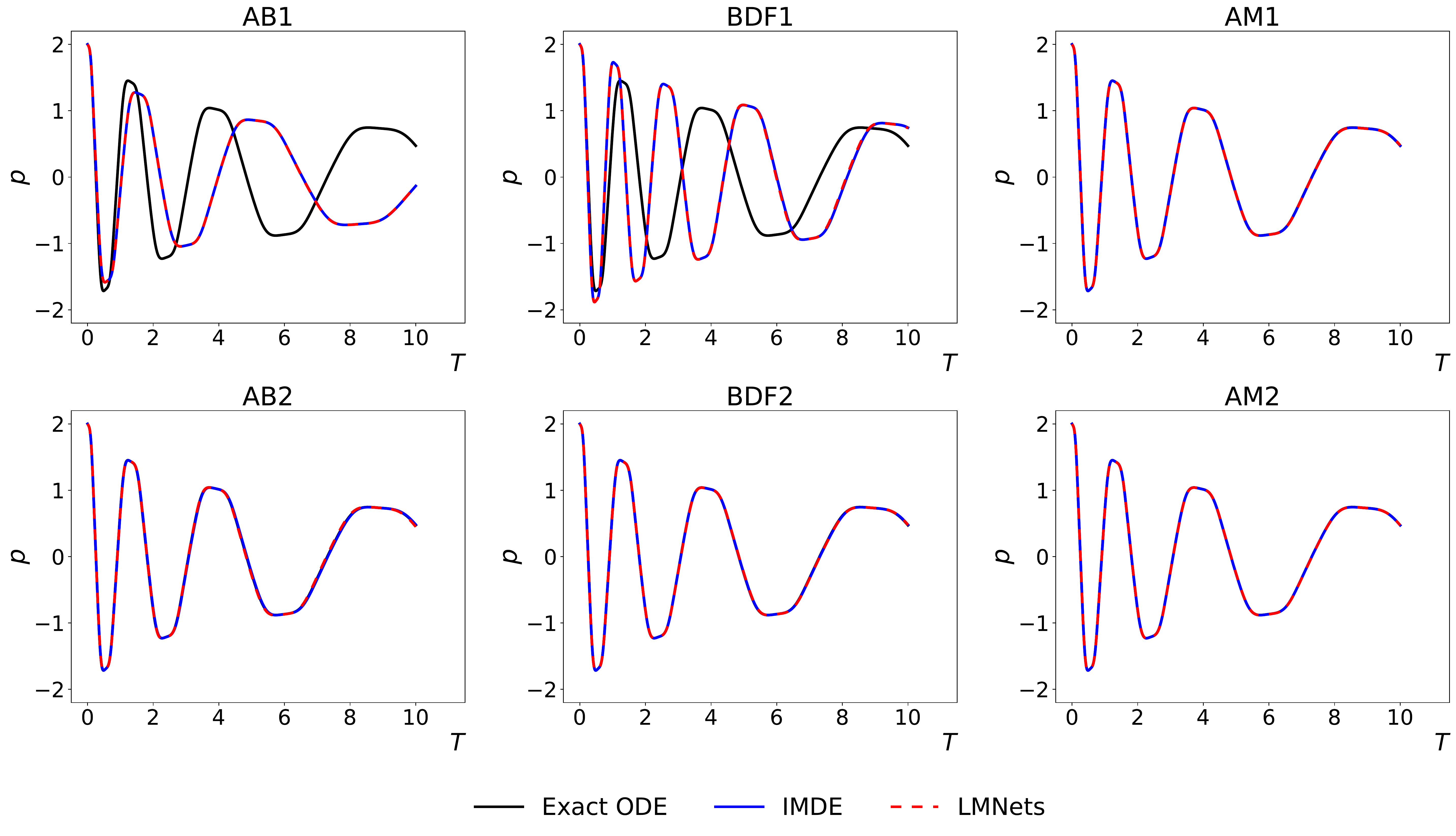}
  \caption{The first component of the trajectories of exact, modified and discovered systems for the damped harmonic oscillator (\ref{eq:do}).}
  \label{fig:DO_traj}
\end{figure}

We generate $1500$ trajectories for the trainingdata, where each trajectory consists of an initial state randomly sampled from $[-2.2,2.2]\times[-2.2,2.2]$ and $10$ subsequent states at equidistant time steps $h$. The neural networks employed consist of two hidden layers with 128 units and \texttt{tanh} activation. We optimize the training loss (\ref{eq:loss-of-lmnets}) using a full batch for $10^5$ epochs, where the optimizer is chosen to be Adam optimization \cite{kingma2014adam} in the Pytorch library, and the learning rate is set to decay exponentially with linearly decreasing powers from $10^{-2}$ to $10^{-4}$.

We first verify the convergence rate with respect to the step size $h$. We randomly sample $10^4$ points from $ [-2,2]\times[-2,2]$ as $\mathcal{T}$, and evaluate the error between $f_{\theta}$ and $f$, as well as the error between $f_{\theta}$ and $f_h^4$ according to (\ref{eq:error of experiments}). The error versus steps $h$ (assigned as $h=0.002i, i=1, \cdots, 8$) and schemes (AB, BDF, AM) is depicted in \cref{fig:DO}, where $5$ independent experiments are performed for each case to obtain the means and the standard deviations. It is shown in \cref{fig:DO} that the error order is consistent with the order of the employed LMM when the discretization error is dominated (see AB1, BDF1 and AB2, BDF2, AM1 with large $h$), which supports the theoretical analysis.

In addition, \cref{fig:DO} also shows that the error between $f_{h}$ and $f_{\theta}$ is significantly lower than that between $f$ and $f_{\theta}$ despite employing only truncation $f_h^4$.  Additionally, we simulate the discovered system $\frac{d}{dt}y = f_{\theta}(y)$, the exact system (\ref{eq:do}), and the truncated IMDE $\frac{d}{dt}y = f_{h}^4(y)$ starting at $(2,0)$. The first components of these trajectories are plotted in \cref{fig:DO_traj} for comparison. It can be observed that the discovered dynamical systems accurately capture the evolution of the corresponding IMDE when using LMMs of order 1. And LMNets successfully learn the target system when the IMDE is close to the target system (AB2, BDF2, AM1 and AM2). These results indicate that training LMNets returns an approximation of the IMDE, which is consistent with the theoretical conclusions in \cref{the:error}.

{
\begin{figure}[t]
  \centering
  \includegraphics[width=1\textwidth]{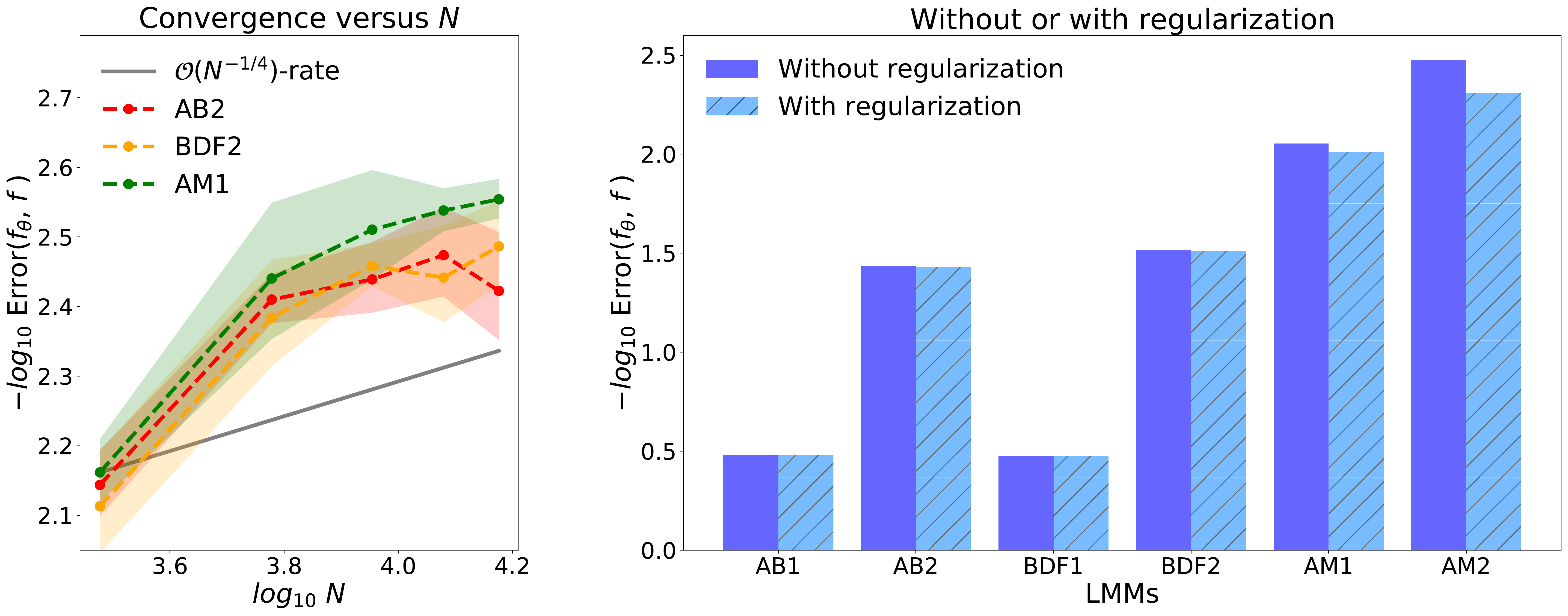}
  \caption{(Left) $\text{Error}(f_{\theta}, f)$ versus the number of training data samples $N$ for the damped harmonic oscillator (\ref{eq:do}). (Right) $\text{Error}(f_{\theta}, f)$ for the damped harmonic oscillator (\ref{eq:do}) without or with regularization. (Total figure) The height of the bars stands for $-log_{10}$ of the $\text{Error}(f_{\theta}, f)$; the higher the better. The results are obtained by taking the mean of five independent experiments, and the shaded region on the left side represents one standard deviation.}
  \label{fig:DO2}
\end{figure}

Subsequently, we investigate the convergence rate with respect to the number of training data samples by employing different numerical schemes (AB2, BDF2 and AM1) and taking $h=0.002$. As shown in \cref{fig:DO}, these assignments have a small discretization error and are appropriate for the test. In the left of \cref{fig:DO2}, the errors are plotted with respect to the number of trajectories, ranging from $300$ to $1500$ (on each trajectory, there are $9$ groups of training samples for AB2 and BDF2, and $10$ groups of training samples for AM1). As expected by \cref{the:error2}, the error between $f$ and $f_{\theta}$ convergence to zero as the number of training data samples increases. However, we observe that the rates of convergence are higher than the theoretical analysis. Further research is necessary to identify the factors that may influence the convergence rates when discovering dynamical systems using deep learning.

Finally, we would like to investigate the influence of the explicit regularization. For various schemes and a fixed step $h=0.01$, we optimize the regularized loss (\ref{eq:regularized training loss}) and record the error in the right of \cref{fig:DO2}. It is shown that adding explicit regularization does not improve the performance, indicating the implicit regularization is effective enough, and explicit regularization may not be necessary. Thus we still use the non-regularized loss (\ref{eq:loss-of-lmnets}) and rely on implicit regularization below.
}

\subsection{Lorenz system}\label{sec:ls}
Subsequently, we consider the normalized 3-D Lorenz system \cite{lorenz1963deterministic,sparrow2012lorenz}, which is a simplified model for atmospheric dynamics and is given as
\begin{equation}\label{eq:LS}
\begin{aligned}
\frac{d}{dt}p =10(q-p),\quad
\frac{d}{dt}q =p(28-10r)-q,\quad
\frac{d}{dt}r =10pq-\frac{8}{3}r,
\end{aligned}
\end{equation}
where $y=(p,q,r)$. Here, the training data consists of data points
on a single trajectory from $t = 0$ to $t = 10$ with data step
$h$ and initial condition $y_0 = (-0.8, 0.7, 2.7)$,
\begin{equation*}
\mathcal{T}_{train} = \left\{\big(x_n, \phi_{h,f}(x_n), \cdots, \phi_{Mh,f}(x_n)\big)| x_n= \phi_{nh}(y_0), n=0,1,\cdots, 10/h-M \right\},
\end{equation*}
and the test data is set as 
\begin{equation*}
\mathcal{T}_{test} = \left\{\big(x_n, \phi_{h,f}(x_n), \cdots, \phi_{Mh,f}(x_n)\big)| x_n= \phi_{n\hat{h}}(y_0), n=0,1,\cdots, (10-Mh)/\hat{h} \right\},
\end{equation*}
where $\hat{h}=0.002$. The chosen model architecture and hyper-parameters are the same as in \cref{sec:do} except that the total number of epochs is set as $3\times 10^5$.

\begin{table}[ht]
  \centering
  \caption{Quantitative results of the Lorenz system (\ref{eq:LS}) obtained by AB and BDF schemes. 
  }
  \begin{tabular}{|c|c|ccc|ccc|}
  \hline
\multicolumn{2}{|c|}{Schemes} & \multicolumn{3}{c|}{AB} & \multicolumn{3}{c|}{BDF}\\
\hline
$M$ & $h$ & Test Loss$^{\frac{1}{2}}$ & Error& Order & Test Loss$^{\frac{1}{2}}$ & Error& Order\\
\hline
 \multirow{2}{*}{$1$} &0.002 & 1.757e-03 & 1.709e-01 & ---& 1.777e-03 & 1.710e-01 & ---\\
&0.004 & 1.837e-03 & 3.418e-01 & 1.000 & 1.769e-03 & 3.419e-01 & 1.000 \\
&0.008 & 2.072e-03 & 6.833e-01 & 0.999 & 2.190e-03 & 6.836e-01 & 1.000 \\
&0.016 & 1.979e-02 & 1.365e+00 & 0.999 & 1.570e-02 & 1.366e+00 & 0.999 \\
&0.032 & 1.404e-01 & 2.706e+00 & 0.987 & 1.139e-01 & 2.730e+00 & 0.999 \\
\hline
 \multirow{2}{*}{$2$}&0.002 & 1.788e-03 & 6.856e-03 & ---& 1.901e-03 & 6.305e-03 & ---\\
&0.004 & 1.812e-03 & 2.047e-02 & 1.578 & 1.804e-03 & 1.695e-02 & 1.427 \\
&0.008 & 2.034e-03 & 7.849e-02 & 1.939 & 1.946e-03 & 6.406e-02 & 1.918 \\
&0.016 & 2.095e-02 & 3.047e-01 & 1.957 & 1.526e-02 & 2.546e-01 & 1.991 \\
&0.032 & 1.347e-01 & 1.101e+00 & 1.853 & 1.531e-01 & 9.805e-01 & 1.945 \\
\hline
 \multirow{2}{*}{$3$} &0.002 & 1.793e-03 & 4.179e-03 & ---& 1.803e-03 & 4.301e-03 & ---\\
&0.004 & 1.843e-03 & 5.276e-03 & 0.336 & 1.795e-03 & 5.237e-03 & 0.284 \\
&0.008 & 2.247e-03 & 1.836e-02 & 1.799 & 1.913e-03 & 1.492e-02 & 1.511 \\
&0.016 & 3.037e-02 & 1.209e-01 & 2.719 & 1.893e-02 & 9.214e-02 & 2.626 \\
&0.032 & 1.537e-01 & 7.597e-01 & 2.651 & 1.500e-01 & 5.622e-01 & 2.609 \\
\hline
  \end{tabular}
  \label{tab:LSerror}
\end{table}
\begin{table}[ht]
  \centering
  \caption{Quantitative results of the Lorenz system (\ref{eq:LS}) obtained by AM schemes.
  }
  \begin{tabular}{|c|ccc|ccc|}
  \hline
{Schemes} & \multicolumn{3}{c|}{AM1} & \multicolumn{3}{c|}{AM2}\\
\hline
$h$ & Test Loss$^{\frac{1}{2}}$ & Error& Order & Test Loss$^{\frac{1}{2}}$ & Error& Order\\
\hline
0.002 & 1.817e-03 & 4.405e-03 & ---& 1.865e-03 & 4.310e-03 & ---\\
0.004 & 1.823e-03 & 6.232e-03 & 0.501 & 1.789e-03 & 4.331e-03 & 0.007 \\
0.008 & 1.888e-03 & 1.704e-02 & 1.451 & 1.978e-03 & 6.074e-03 & 0.488 \\
0.016 & 1.524e-02 & 6.796e-02 & 1.996 & 1.217e-02 & 2.486e-02 & 2.033 \\
0.032 & 1.250e-01 & 2.948e-01 & 2.117 & 1.210e-01 & 1.849e-01 & 2.894 \\
  \hline
  \end{tabular}
  \label{tab:LSerrorAM}
\end{table}
\begin{table}[!ht]
  \centering
  \caption{Leading terms of $f_h-f$.
  }
  \begin{tabular}{|c|ccc|ccc|}
  \hline
{Schemes} & AB2 & BDF2 & AM1 & AB3 &BDF2 &AM2\\
\hline
\rule{0pt}{12pt}
Leading terms & $\frac{5h^2}{12} \D^2f$ & $-\frac{h^2}{3} \D^2f$& $-\frac{h^2}{12} \D^2f$ & $\frac{3h^3}{8} \D^3f$ & $-\frac{h^3}{4} \D^3f$& $-\frac{h^3}{24} \D^3f$\\
  \hline
  \end{tabular}
  \label{tab:leading term}
\end{table}

We first test the convergence with respect to data step size $h$ for AB and BDF schemes. Herein, we evaluate the test loss on test data to measure the learning loss a posteriori, and record the test loss as well as the error evaluated on the test data in \cref{tab:LSerror} to show the convergence more clearly. Here, $5$ independent experiments are simulated to obtain the means. As discussed and numerically verified in \cite{du2022discovery}, LMNets without the auxiliary conditions perform well when $h$ is small. Thus we only take relatively small $h$. It is shown in \cref{tab:LSerror} that there exists a threshold caused by the learning loss, and LMNets can effectively discover the hidden dynamics with an error order consistent with the theoretical ones when the error is significantly larger than the square root of the test loss.

Moreover, we test the performance of AM schemes using the same settings and report the quantitative results in \cref{tab:LSerrorAM}. Despite that LMNets with AM are unstable for discovery \cite{du2022discovery}, they can still perform well and the same error phenomenon can be observed, which is also consistent with the theoretical analysis. It is observed that the error of LMNets using AM is smaller than that using AB and BDF schemes of the same order, which is due to the different leading terms of $f_h-f$ (obtained by \cref{cor:fh-f} and presented in \cref{tab:leading term}).

\subsection{Glycolytic oscillator}
Finally, we consider the model of oscillations in yeast glycolytic. The model describes the concentrations of 7 biochemical species and is formulated as
\begin{equation}\label{eq:go}
\begin{aligned}
\frac{d}{dt}S_1 =& J_0 - \frac{k_1 S_1S_6}{1+(S_6/K_1)^q},\\
\frac{d}{dt}S_2 =& 2 \frac{k_1 S_1S_6}{1+(S_6/K_1)^q} - k_2S_2(N-S_5)-k_6S-2S_5,\\
\frac{d}{dt}S_3 =& k_2S_2(N-S_5) - k_3S_3(A-S_6),\\
\frac{d}{dt}S_4 =& k_3S_3(A-S_6) - k_4S_4S_5 - \kappa(S_4-S_7),\\
\frac{d}{dt}S_5 =& k_2S_2(N-S_5) - k_4S_4S_5 - k_6S_2S_5,\\
\frac{d}{dt}S_6 =& -2 \frac{k_1 S_1S_6}{1+(S_6/K_1)^q} +2k_3S_3(A-S_6)-k_5S_6,\\
\frac{d}{dt}S_7 =&\psi \kappa (S_4-S_7) - kS_7,
\end{aligned}
\end{equation}
where $y=(S_1, \cdots, S_7)$ and the parameters are taken from Table 1 in \cite{daniels2014efficient}. We simulate $60$ and $40$ trajectories ($T=10$) for the training and test data, respectively. For each trajectory, the initial state is randomly sampled from half of the ranges provided in table 2 of \cite{daniels2014efficient}. The chosen model architecture and the other experimental setup are the same as in \cref{sec:ls} except the batch size is set as $3\times 10^4$. 
\begin{figure}[ht]
  \centering
  \includegraphics[width=1\textwidth]{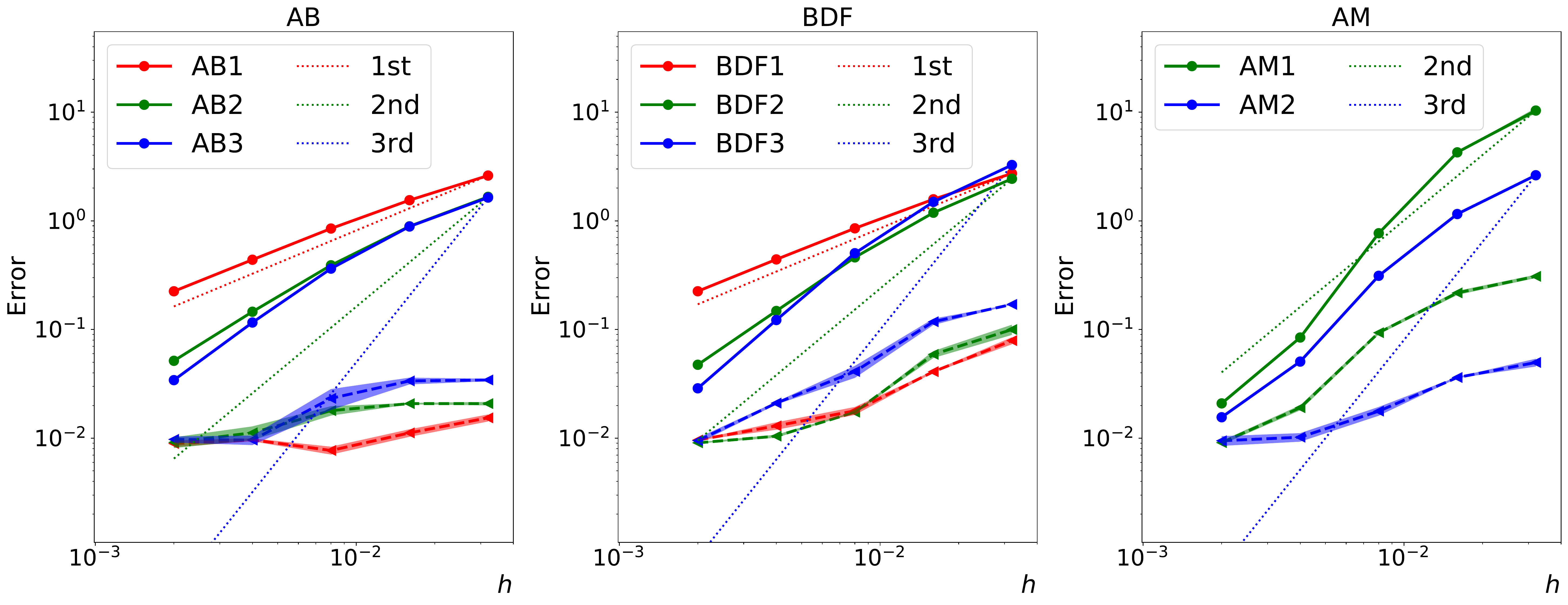}
  \caption{$\text{Error}(f_{\theta}, f)$ (solid curves with circle markers) and the square root of test loss (dashed curves with triangle markers) versus $h$ for the glycolytic oscillator (\ref{eq:go}). The results are obtained by taking the mean of $5$ independent experiments, and the shaded region represents one standard deviation.}
  \label{fig:GO}
\end{figure}
\begin{figure}[ht]
  \centering
  \includegraphics[width=0.9\textwidth]{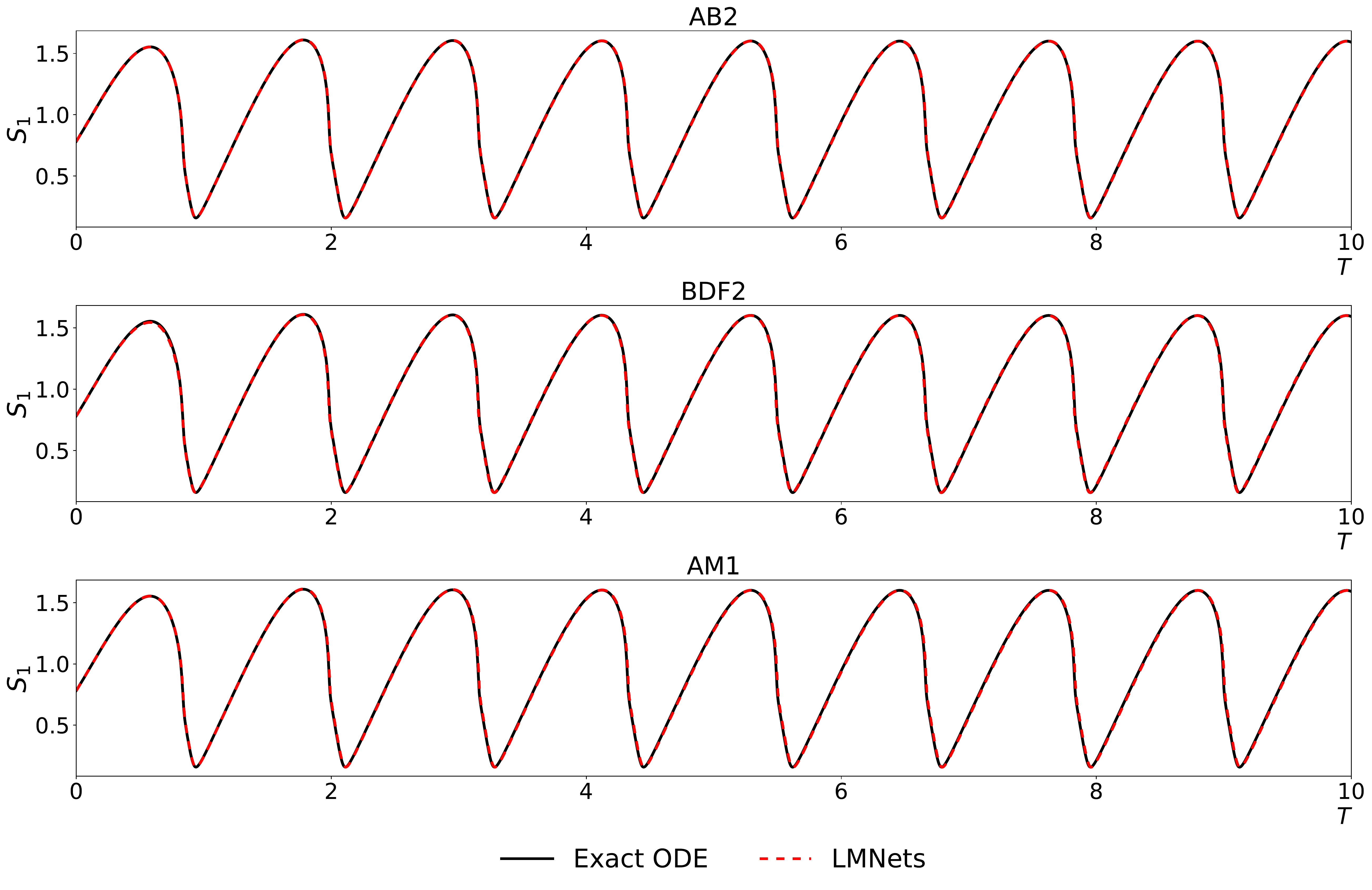}
  \caption{The first component of the trajectories of the exact system (\ref{eq:go}) and the discovered systems for the glycolytic oscillator (\ref{eq:go}).}
  \label{fig:GOtraj}
\end{figure}

We continue to test the convergence of LMNets using AB, BDF and AM schemes with respect to $h$ (assigned as $h=0.002*2^i, i=0,\cdots,4$), where the error is evaluated on the test data. The error and the test loss versus $h$ is shown in \cref{fig:DO}. It is shown that the convergence rates of all schemes are lower than the theoretical ones when $h$ is relatively large. These observations are in agreement with the numerical results in \cite{du2022discovery}, which are caused by the low regularity of the glycolytic oscillator (\ref{eq:go}). Compared to the stable schemes AB, BDF, it is observed that the unstable schemes AM can also learn hidden dynamics and are more accurate when $h$ is relatively small due to the smaller leading terms of $f_h-f$ (given in \cref{tab:leading term}). However, when $h$ is relatively large, due to the unstable properties when used for discovery \cite{du2022discovery}, it is difficult to train and the errors would be large. 

We also perform predictions for the glycolytic oscillator. We use the first trajectory in the test data as the baseline and simulate the discovered systems using the same initial condition, where the test schemes are AB(M = 2), BDF(M = 2), and AM(M = 1), and the data step size is $h = 0.002$. The first components of the trajectories are depicted in \cref{fig:GOtraj} for comparison. It is observed that all schemes can correctly capture the form of the dynamics.

\section{Summary}\label{sec: summary}
We provide an error analysis for the discovery of dynamics using linear multistep methods with deep learning in this paper. Our study reveals that such a learning model returns a close approximation of the IMDE. We thus establish that the error between the discovered system and the target dynamical system is bounded by the sum of the LMM discretization error $\mathcal{O}(h^p)$ and the learning loss. In addition, a priori error bounds is given following the approximation and generalization theories of neural networks. Our conclusions (i) do not depend on auxiliary conditions; (ii) hold for any weakly stable and consistent LMM; (iii) and are able to quantify the error out of the sample locations. Moreover, numerical results support the theoretical analysis. 

The analysis in this paper is restricted to the discovery of low-frequency dynamics with fine data steps. For relatively large data step, this paper does not improve the existing stability analysis and error estimation \cite{keller2021discovery,du2022discovery}, which still remain unknown without the help of auxiliary conditions. For the discovery of high-frequency dynamics, we may need to design specific discretization schemes to construct loss functions, and develop corresponding error analysis. One possible direction is the filtered integrator and the modulated Fourier expansion \cite{hairer2001long}.

The analyticity assumption of the learned model forms the basis of our analysis, but it is proved only for the regularized model. Our numerical experiments rely on implicit regularization, and show that discovery of dynamics using deep learning performs well without regularization. Since the analyticity assumption and implicit regularization are closely related, and there is a lack of general theoretical evidence and precise descriptions of the underlying mechanism of implicit regularization, we would like to investigate these assumptions further in future research, specifically from the perspective of the frequency principle/spectral bias of neural networks.


\bibliographystyle{abbrv}
\bibliography{references}

\end{document}